\documentclass[12pt]{amsart}
\usepackage{ amsmath, tabularx, amsthm, amssymb, amsfonts}

\usepackage[margin=1in]{geometry}

\usepackage{setspace}
\usepackage{color,xcolor}
\definecolor{chianti}{rgb}{0.6,0,0}
\definecolor{meretale}{rgb}{0,0,.6}
\definecolor{leaf}{rgb}{0,.35,0}
\usepackage[colorlinks=true, pagebackref, hyperindex, citecolor=meretale, urlcolor=leaf, linkcolor=chianti]{hyperref}

\usepackage[all]{xy}  
\usepackage{comment}
\usepackage{enumerate}
\usepackage{mathrsfs}              
\usepackage{stmaryrd}
\usepackage{colonequals}
\usepackage{xfrac} 
\usepackage{todonotes}

\theoremstyle{definition}
\newtheorem{theorem}{Theorem}[section]

\newtheorem{corollary}[theorem]{Corollary}
\newtheorem{lemma}[theorem]{Lemma}
\newtheorem{proposition}[theorem]{Proposition}

\newtheorem{notation}[theorem]{Notation}

\theoremstyle{definition}

\newtheorem{definition}[theorem]{Definition}

\newtheorem{example}[theorem]{Example}

\newtheorem{remark}[theorem]{Remark}
\numberwithin{equation}{subsection}

\newcommand{\m}{\mathfrak{m}}

\newcommand{\RR}{\mathbb{R}}
\newcommand{\NN}{\mathbb{N}}
\renewcommand{\P}{\mathbb{P}}
\newcommand{\N}{\mathbb{N}}
\newcommand{\C}{\mathbb{C}}
\newcommand{\ZZ}{\mathbb{Z}}
\newcommand{\Z}{\mathbb{Z}}

\newcommand{\F}{\mathbb{F}}

\newcommand{\cF}{\mathcal{F}}

\newcommand{\cA}{\mathcal{A}}
\newcommand{\cB}{\mathcal{B}}
\newcommand{\cC}{\mathcal{C}}
\newcommand{\cD}{\mathcal{D}}
\newcommand{\cJ}{\mathcal{J}}
\newcommand{\cBb}{\mathcal{B}^\bullet}
\newcommand{\cFb}{\mathcal{F}^\bullet}
\newcommand{\cGb}{\mathcal{G}^\bullet}

\newcommand{\BS}[2]{\mathrm{BS}^{#2}_{#1}}
\newcommand{\SBS}[2]{\mathrm{SBS}^{#2}_{#1}}

\newcommand{\Dim}{\operatorname{Dim}}
\newcommand{\ee}{\operatorname{e}}

\newcommand{\Spec}{\operatorname{Spec}}

\newcommand{\Hom}{\operatorname{Hom}}
\newcommand{\End}{\operatorname{End}}

\newcommand{\fs}{\boldsymbol{f^s}}

\newcommand{\Q}{\mathbb{Q}}

\newcommand{\e}{\operatorname{e}}

\newcommand{\op}{\mathrm{op}}
\newcommand{\Der}{\mathrm{Der}}
\def\ord{\operatorname{ord}}
\newcommand{\Frac}{\mathrm{Frac}}

\newcommand{\mSpec}{\operatorname{mSpec}}

\author[Jeffries]{Jack Jeffries}
\address{Department of Mathematics, University of Nebraska, Lincoln, NE 68588-0130, USA}
\email{jack.jeffries@unl.edu}
\thanks{JJ and DL were partially supported by NSF CAREER Award DMS-2044833.}

\author[Lieberman]{David Lieberman}
\address{Department of Mathematics, University of Nebraska, Lincoln, NE 68588-0130, USA}
\email{david.lieberman@huskers.unl.edu}

\begin{document}
\title[Sandwich Bernstein-Sato Polynomials]{Sandwich Bernstein-Sato polynomials and Bernstein's inequality}

\maketitle

\begin{abstract} Bernstein's inequality is a central result in the theory of $D$-modules on smooth varieties. While Bernstein's inequality fails for rings of differential operators on general singularities, recent work of \`{A}lvarez Montaner, Hern\'andez, Jeffries, N\'u\~nez-Betancourt, Teixeira, and Witt establishes Bernstein's inequality for invariants of finite groups in characteristic zero and certain other mild singularities in positive characteristic. Motivated by extending this result to new classes of singular rings, we introduce a ``two-sided'' analogue of the Bernstein-Sato polynomial which we call the sandwich Bernstein-Sato polynomial. We apply this notion to give an effective criterion to verify Bernstein's inequality, and apply this to show that Bernstein's inequality holds for the coordinate ring of $\mathbb{P}^a \times \mathbb{P}^b$ via the Segre embedding. We also establish a number of examples and basic results on sandwich Bernstein-Sato polynomials.
\end{abstract}

\section{Introduction}

Recently, a number of key results from the theory of $D$-modules on smooth varieties have been extended to $D$-modules over sufficiently mild singularities. Throughout this introduction, let $R$ be an algebra over a field $K$ of characteristic zero, and $D_R$ the ring of $K$-linear differential operators in the sense of Grothendieck; when $R$ is a polynomial ring, this is the familiar \emph{Weyl algebra} generated over $R$ by the partial derivatives with respect to the variables. 

A {\textit{Bernstein-Sato functional equation}} for an element $f\in R$ is an equation of the form
\[ \delta(s) \cdot f^{s+1} = b(s) f^s\]
for some $\delta(s)\in D_R[s]$ and $b(s)\in K[s]$. The existence of a nonzero Bernstein-Sato functional equation for any nonzero element when $R$ is smooth is one of the fundamental results in the theory of $D$-modules, and the study of the minimal monic $b(s)$ arising in an equation as above for a fixed $f$ (i.e., the \textit{Bernstein-Sato polynomial of $f$}) has played an important role in singularity theory and algebraic geometry \cite{BernsteinPoly,Mal74,Mal83,ELSV04,MTW}. Work of \`Alvarez Montaner, Huneke, and N\'u\~nez-Betancourt \cite{AMHNB} extended this existence result to $K$-algebras that are direct summands of polynomial rings, a class of rings that includes invariant rings of linearly reductive groups.

 {\textit{Bernstein's inequality}} in the theory of $D$-modules asserts that every nonzero $D_R$-module over a smooth $K$-algebra $R$ has dimension at least $\dim(R)$, for a suitable sense of dimension akin to the Krull dimension of a module over a commutative ring. This result has many strong consequences; in particular, it easily implies the existence of Bernstein-Sato polynomials, and that local cohomology modules of $R$ have finite length as $D_R$-modules. Recent work of \`Alvarez Montaner, Hern\'andez, Jeffries, N\'u\~nez-Betancourt, Teixeira, and Witt extended this result to rings of invariants of finite groups, as well as a wider class of algebras in positive characteristic \cite{AHJNTW2}. 
 
The work in this paper is motivated by extending Bernstein's inequality to new classes of singular rings and studying the relationship between Bernstein's inequality and Bernstein-Sato theory. These considerations motivate the following definition: a  {\textit{sandwich Bernstein-Sato functional equation}} for an element $f\in R$ is an equation of the form
\[ \sum \alpha_i(s)  f^{s+1} \beta_i(s) = b(s) f^s\]
for some $\alpha_i(s), \beta_i(s)\in D_R[s]$ and $b(s)\in K[s]$. Here we consider $f$ as a differential operator of order zero and regard this as an equality of \emph{differential operators} rather than of ring elements under the module action.

Our main result applies this notion of sandwich Bernstein-Sato equation to give an effective criterion for $R$ to satisfy Bernstein's inequality.

\begin{theorem}[Corollary~\ref{thm-main-BI}, cf.~Remark~\ref{rem:BIintro}]\label{thm1.1} Let $R$ be a finitely generated $\N$-graded algebra over a field $K$ of characteristic zero. Suppose that there exists some nonzerodivisor $f$ in the singular locus of $R$ that admits a sandwich Bernstein-Sato polynomial with no nonnegative integer roots. Then Bernstein's inequality holds for $D_R$:~every nonzero $D_R$-module $M$ has dimension at least $\dim(R)$.
\end{theorem}

We apply our criterion to show that coordinate rings of Segre embeddings of the product of projective spaces satisfy Bernstein's inequality.

\begin{theorem}[Theorem~\ref{thm:Segre}]
Let $R$ be the coordinate ring of $\P^a \times\P^b$ via the Segre embedding. Then Bernstein's inequality holds for $D_R$:~every nonzero $D_R$-module $M$ has dimension at least $\dim(R)$.
\end{theorem}

A main technical tool used in this work is the notion of  {\textit{linear simplicity}} of a filtered algebra arising in work of Bavula and Lyubeznik \cite{Bav96,Bav09,Lyu11}; this is a condition on the algebraic structure of a filtration of $D_R$ that is a sharpening of the condition that $D_R$ is a simple ring. For a graded algebra $R$, there is a family of filtrations on $D_R$ that we refer to as \emph{Bernstein filtrations} that generalize the Bernstein filtration on the ring of differential operations on a polynomial ring (see Subsubsection~\ref{sssec:BF}); linear simplicity of this filtration, with mild additional hypotheses, implies Bernstein's inequality (see Subsection~\ref{ssec:BI} below). It remains open whether rings of differential operators are linearly simple (with respect to the Bernstein filtration) for all graded direct summands of polynomial rings; a positive answer would imply a well-studied conjecture of Levasseur, Stafford, and Schwarz \cite{LS,Schwarz} on simplicity of rings of differential operators for rings of invariants of linearly reductive groups in characteristic zero.

We obtain the following strong relationship between the existence of sandwich Bernstein-Sato functional equations and linear simplicity:
\begin{theorem}[Theorems~\ref{thm:LSimpliesSBS} and~\ref{SBSimpliesLS} and Corollary~\ref{cor:toric}]
    Let $R$ be a finitely generated $\N$-graded $K$-algebra, with $R_0=K$ a field of characteristic zero.
    \begin{enumerate}
        \item If there exists a nonzerodivisor $f$ in the singular locus of $R$ that admits a sandwich Bernstein-Sato polynomial with no nonnegative integer roots, then $D_R$ is linearly simple with respect to any Bernstein filtration.
        \item If $R$ is an affine semigroup ring and $D_R$  with respect to any Bernstein filtration is linearly simple, then every nonzero element of $R$ admits a nonzero sandwich Bernstein-Sato polynomial.
        \item[(2$^\prime$)] More generally, if $D_R\otimes_K D_R^{\op}$ with respect to any Bernstein filtration is linearly simple, then every nonzero element of $R$ admits a nonzero sandwich Bernstein-Sato polynomial.
    \end{enumerate}
\end{theorem}

We establish a few basic results in the study of sandwich Bernstein-Sato polynomials in their own right. In particular, we show:

\begin{theorem}[Propositions~\ref{prop-divides}~and~\ref{prop-equal} and Example~\ref{ex-f=y}]
Let $f\in R$. If $f$ has a nonzero sandwich Bernstein-Sato functional equation then $f$ admits a Bernstein-Sato polynomial, and the Bernstein-Sato polynomial of $f$ divides the sandwich Bernstein-Sato polynomial. If $R$ is smooth, these two polynomials coincide, though in general they do not.
\end{theorem}

Moreover, equality between sandwich Bernstein-Sato polynomials and Bernstein-Sato polynomials holds whenever the ring of differential operators satisfies the hypothesis of Nakai's conjecture; i.e., $D_R$ is generated by operators of order at most one.

We compute a number of example of sandwich Bernstein-Sato polynomials in Section~6. Finally, we note that our methods yield a result on the numerical invariant introduced in \cite{BJNB} called differential signature:

\begin{theorem}[Theorem~\ref{thm:diffsig}]
    Let $R$ be a graded ring. If there exists an element $f$ in the singular locus of $R$ that admits a Bernstein-Sato polynomial with no nonnegative integer roots, then the differential signature of $R$ is greater than zero.
\end{theorem}

For the reader's convenience, we end this introduction with an index of notation and terminology that will be used throughout the paper.

%\section*{Index of some notation and terminology}

\begin{tabular}{lll}
$D_{R|K}$, $D$ & ring of $K$-linear differential operators & Definition~\ref{def:diffops} \\
$D^i_{R|K}$, $D^i$ &differential operators of order at most $i$  & Definition~\ref{def:diffops} \\
$\ord(\delta)$ & order of a differential operator & Definition~\ref{def:diffops} \\
$[\delta]^{i_1,\dots,i_t}_{f_1,\dots,f_t}$, $[\delta]^{\underline{i}}_{\underline{f}}$ & iterated commutator bracket & Notation~\ref{not:brac}\\
$\deg(\delta)$ & degree of a homogeneous differential operator & Definition~\ref{degree of an operator} \\
$[D]_i$, $[D]_{\geq i}$ & operators of degree $i$ or $\geq i$ (resp.) & Definition~\ref{degree of an operator} \\

$\cFb$, $\cGb$ & filtrations on a ring or module & Defs~\ref{def:filt1},~\ref{def:filt2} \\
$\cBb$ & Bernstein filtration on $D_{R|K}$ for  $R$ graded & Definition~\ref{def:Ber filt}\\
$B^i$, $F^i$, $G^i$ & filtered component of $\cBb$, $\cFb$, $\cGb$ (resp.) & Definition~\ref{def:filt2}\\
$\mathcal{F}^\bullet_{\boxtimes}$ & filtration on $A \otimes A^\mathrm{op}$ induced by $\cFb$ on $A$ & Remark~\ref{rem:box}\\
$\Dim(A,\cFb), \Dim(\cFb)$ & dimension of filtration $\cFb$ on $A$ & Definition~\ref{def:dim} \\
$\ee(A,\cFb), \ee(\cFb)$ & multiplicity of filtration $\cFb$on $A$  & Definition~\ref{def:dim} \\
$\mathrm{s}^{\mathrm{diff}}_K(R)$ & differential signature & Definition~\ref{def:diffsig} \\
$\m^{\langle n \rangle_K}$ & differential powers & Definition~\ref{def:diffsig} \\
$\SBS{f}{R}(s)$ & sandwich Bernstein-Sato poly.~of $f\in R$ & Definition~\ref{def:sbspoly}\\
$\BS{f}{R}(s)$ & Bernstein-Sato polynomial of $f\in R$ & Remark~\ref{rem:standardBS}\\
& linearly simple & Definition~\ref{def:lin simple}\\
& generator filtration & Definition~\ref{def:gen filt}\\
& localized filtration & Definition~\ref{def:loc filt}\\

\end{tabular}

\section{Background}

\subsection{Rings of differential operators}

\begin{definition}\label{def:diffops}
Let $R$ be an $A$-algebra, where $R$ and $A$ are commutative rings. %Let $[-,-]$ denote the commutator operation. 
The  {\textit{ring of differential operators}} on $R$ is a subring of $\Hom_R(R,R)$ defined inductively (by order) as follows:
    $$ D^0_{R|A} = \Hom_R(R,R), \quad \text{and}$$
    $$ D^{n+1}_{R|A} = \{\delta \in \Hom_A(R,R)\,|\,[\delta,\overline{r}]\in D^{n}_{R|A} \forall r\in D^0_{R|A} \},$$
    where  $[-,-]$ denotes the commutator. The union
    \[ D_{R|A} = \bigcup_{n=0}^{\infty} D^n_{R|A}\]
    is a ring with composition as the multiplication operation.

    For an element $\delta\in D_{R|A}$, the \textit{order} of $\delta$, denoted $\ord(\delta)$, is the smallest $n$ such that $\delta\in D^n_{R|A}$.
\end{definition}

It will often be useful to repeatedly apply the commutator, and thus we adopt the following iterative notation: 
    
\begin{notation}\label{not:brac}
    For any positive integer $n$ and operators $\delta$ and $\beta$, we denote
    \begin{itemize}
        \item $[\delta]^0_\beta \colonequals \delta$
        \item $[\delta]^1_\beta \colonequals [\delta,\beta]$
        \item $[\delta]^{n}_\beta \colonequals \big[[\delta]^{n-1}_\beta,\beta\big]$
        \item $[\delta]^{\underline{i}}_{{\underline{\beta}}} \colonequals [\cdots[[\delta]^{i_1}_{\beta_1}]^{i_2}_{\beta_2}\cdots]^{i_n}_{\beta_n}$ for any sequence of (not necessarily distinct) operators $\underline{\beta} = \beta_1,\ldots,\beta_n$ and any sequence of natural numbers $\underline{i} = i_1,\ldots,i_n.$
    \end{itemize}
\end{notation}

\begin{example}\label{Weyl Algebra}
Let $K$ be a field of characteristic zero, and consider the polynomial ring $R=K[x_1,\ldots,x_n]$. The ring of differential operators $D_{R|A}$ is the {\textit{Weyl algebra}}, which is the noncommutative algebra $K\langle x_1,\ldots,x_n, \partial_{x_1},\ldots,\partial_{x_n}\rangle$, where $\partial_{x_i} \colonequals \sfrac{\partial}{\partial x_i}$. Equivalently, this is the $K$-algebra generated by the indeterminates and the derivations of $R$. 
\end{example}

\begin{definition}\label{degree of an operator} Let $D$ be the ring of differential operators on $R$ where $R$ is a $\Z$-graded $A$-algebra. A differential operator has  {\textit{degree}} $n$ if it is a degree $n$ map in $\Hom_A(R,R)$; i.e., $\delta \in D_{R|A}$ has degree $n$ if and only if $\deg(\delta(f)) = \deg(f)+n$ for all homogeneous $f\in R$. We denote the collection of homogeneous elements of degree $i$ by $[D]_i$, and set $[D]_{\geq i} = \sum_{j\geq i} [D]_i$. We will use similar notations for elements of $R$.
\end{definition}

\begin{example}\label{ex Bad Cubic}
 The following example is due to Bernstein-Gelfand-Gelfand \cite{BGG}: Let $R= \C[x,y,z]/(x^3+y^3+z^3)$, and let $D= D_{R|\C}$. In this case $D$ exhibits some interesting properties:
\begin{itemize}
    \item $D$ contains no differential operators of negative degree.
    \item The degree-zero differential operators are generated as a $\C$-vectorspace by the set $\{1,E,E^2,\ldots\}$, where $E = x\cdot\partial_x +y\cdot\partial_y+z\cdot\partial_z$.
    \item $D$ is not Noetherian. One can show that the following forms a nonstabilizing ascending chain of two-sided ideals of $D$ for $k\geq 0$:
        $$ J_k = \sum_{n\geq 0} E^n[D^k]_1 +  [D]_{\geq 2}.$$
\end{itemize}
\end{example}

\begin{example}\label{ex Veronese}
A consequence of a theorem of Kantor \cite{Kantor} gives us the following result: Consider a polynomial ring $S$ in $n$ variables over a field $K$ of characteristic zero. Let $G$ be a group acting on $S$. If $G$ contains no pseudoreflections\footnote{We say that an element $g\in G\smallsetminus\{e_G\}$ is a pseudoreflection if $\mathrm{Fix}_K(g)$ has codimension one in ${\mSpec(S) \cong \overline{K}^n}$.}, then ${D_{S^G|K} \cong \{\delta \in D_{S|K} \ | \ \delta(S^G) \subseteq S^G\}}$.

This result can be used to give a concrete description of the differential operators on the Veronese subrings of a polynomial ring. Namely, let $S = K[x_1,\ldots,x_n]$, with $K$ a field of characteristic zero, and let $T_d$ be the set of all degree $d$ monomials in $S$. We call $K[T_d]$ the \textit{degree $d$ Veronese subring} of $S$, and denote this ring by $S^{(d)}$; this ring can also be described as the subring of $S$ spanned by the homogeneous elements whose degree is a multiple of $d$. As long as $K$ has a primitive $d$th root of unity and $n$ is at least two, $S^{(d)}$ can be realized as the invariant ring of a group action with no pseudoreflections. Thus the differential operators on $S^{(d)}$ are precisely the sums of homogeneous differential operators on $S$ that map any homogeneous element of degree a multiple of $d$ in $S$ to another homogeneous element of degree a multiple of $d$ in $S$.

Concretely, as a $K$-vector space, we have
    \begin{equation}\label{eq:veronese} D_{S^{(d)}|K} = K\cdot\{x_1^{a_1}\cdots x_n^{a_n}\partial_{x_1}^{b_1}\cdots \partial_{x_n}^{b_n}\ |\ \sum_{i=1}^n (a_i-b_i)  = kd \text{ for some } k \in \Z\}.\end{equation}

There exists a field extension $L$ of $K$ that contains a primitive $d$th root of unity. For any $K$, the right-hand side in \eqref{eq:veronese} is verified to be a subset of $D_{S^{(d)}|K}$; then since there is an isomorphism \[\begin{aligned} D_{S^{(d)}|K} &\otimes_K L \cong D_{S^{(d)}\otimes_K L|L} \cong D_{(S\otimes_K L)^{(d)}|L} \\&\cong L \cdot\{x_1^{a_1}\cdots x_n^{a_n}\partial_{x_1}^{b_1}\cdots \partial_{x_n}^{b_n}\ | \ \sum_{i=1}^n (a_i-b_i)  = kd \text{ for some } k \in \Z\},\end{aligned}\]
it follows that \eqref{eq:veronese} holds for $K$ as well. In particular, this holds for any $K$ of characteristic zero.
\end{example}

\begin{example}
 For a field $K$ of characteristic zero, and positive integers $a,b>1$, we can consider the Cartesian product of projective spaces $\P^{a-1} \times \P^{b-1}$ as a projective variety via the Segre embedding into $\P^{ab-1}$. The coordinate ring of this variety is isomorphic to the subring 
\[ R:=K[x_iy_j\, |\, 1\leq i \leq a, 1\leq j\leq b] \subseteq S:= K[x_1,\ldots, x_a, y_1,\ldots, y_b].\] 

Following the work of Musson in \cite{Musson}, we can also find the differential operators of the coordinate ring $R$ of this Segre embedding. Suppose that $K$ has characteristic zero. Then the ring of differential operators on $R$ can be described as
    \[D_{R|K} \cong \frac{\{\delta\in D_{S|K} \ | \ \delta(R) \subseteq R \}}{(\sum_{i=1}^a x_i\partial{x_i} - \sum_{i=1}^b y_i\partial{y_i})} = \frac{K\langle \{ x_i y_j, x_i  \partial_{x_k}, y_j  \partial_{y_\ell}, \partial_{x_k}\partial_{y_\ell} \ | \ 1 \leq i,k\leq a, 1 \leq j,\ell \leq b\}\rangle}{(\sum_{i=1}^a x_i\partial{x_i} - \sum_{i=1}^b y_i\partial{y_i})}.\]
\end{example}

We will need a few basic lemmas on rings of differential operators towards our results in the later sections.

\begin{lemma}\label{lem:bracket-order} Let $R$ be an $A$-algebra generated by a set $S\subseteq R$. 
    \begin{enumerate}
    \item An $A$-linear endomorphism $\delta$ of $R$ is a differential operator of order at most $n$ if and only if $[\delta,s]\in D^{n-1}_{R|A}$ for all $s\in S$.
    \item If $\delta\in D_{R|A}$ has order exactly $n$, then there exists a sequence $s_1,\dots,s_{n}\in S$ such that $[\delta]^{1,\dots,1}_{s_1,\dots,s_n} \in D_{R|A}^0\smallsetminus 0$. 
    \end{enumerate}
    \end{lemma}
    \begin{proof} 
    \begin{enumerate}
    \item Since every element of $R$ is an $A$-linear combination of elements of the form $r=s_1^{a_1}\cdots s_t^{a_t}$ with $s_i\in S$, it suffices to show that under the hypothesis, $[\delta,r]\in D^{n-1}_{R|A}$.  We proceed by induction on $a_1+\cdots +a_t$. If $a_1+\cdots +a_t=0$, we have $r=1$ and $[\delta,r]=0$; otherwise, without loss of generality, we can take $a_1>0$ and write $r=r' s_1$. Then
    \[ [\delta,r] = [\delta,r' s_1] = r' [\delta,s_1] + [\delta,r'] s_1 \in D^{n-1}_{R|A}\]
    by the induction hypothesis.
    \item     We argue by induction on $n$, with the case $n=0$ using the empty sequence. Let ${\delta\in D^{n}_{R|A} \smallsetminus D^{n-1}_{R|A}}$ with $n>0$. By part (1), there is some $s_1\in S$ such that  ${[\delta,s_1]\notin D^{n-2}_{R|A}}$, and we also have $[\delta,s_1]\in D^{n-1}_{R|A}$ by definition. The claim then follows from the induction hypothesis.\qedhere
    \end{enumerate}
    \end{proof}

    We will require a few facts about the behavior of differential operators over well-behaved morphisms. Rather than giving the most general results, we will focus on the case of finitely generated algebras.
    
    \begin{lemma}\label{Localizing Things}
    Let $R$ be a finitely generated algebra over a field $K$. 
    \begin{enumerate}
        \item If $W\subseteq R$ is multiplicatively closed, there is an isomorphism  \[ W^{-1} D^{n}_{R|K} \stackrel{\sim}{\to} D^{n}_{W^{-1}R|K}.\]
        \item If $S$ is a formally \'etale algebra $R$-algebra,
        there is an isomorphism  \[ S\otimes_K D^{n}_{R|K} \stackrel{\sim}{\to} D^{n}_{S|K}.\]
        \item If $L$ is a field extension of $K$, there is an isomorphism
        \[ L\otimes_K D^n_{R|K} \stackrel{\sim}{\to} D^n_{L\otimes_K R| L}.\]
    \end{enumerate}
\end{lemma}
\begin{proof}
    For Part (1), see \cite[Proposition~2.17]{BJNB}. For Part (2), see \cite[Theorem~2.2.10]{Masson}. For the third, write $R_L$ for $L\otimes_K R$ for ease of notation. Note that $R_L \otimes_L R_L \cong L\otimes_K (R\otimes_K R)$, and the maps $R\to R_L$ and $R\otimes_K R \to R_L \otimes_L R_L$ are flat.
    From the short exact sequence
    \[ 0 \to \Delta_{R|K} \to R\otimes_K R \xrightarrow{\text{mult}} R \to 0, \]
    where $\Delta_{R|K}$ denotes the kernel of the multiplication map, we obtain another exact sequence
    \[ 0 \to L\otimes_K \Delta_{R|K} \to R_L \otimes_L R_L \xrightarrow{\text{mult}} R_L \to 0  \] 
    so $\Delta_{R_L|L} \cong L\otimes_K \Delta_{R|K}$; by flatness, we have $\Delta_{R_L|L} = \Delta_{R|K}(R_L \otimes_L R_L)$. Then for each $n$, 
    \[ P^n_{R_L|L} \cong \frac{R_L \otimes_L R_L}{\Delta_{R_L|L}^{n+1}} \cong \frac{R_L \otimes_L R_L}{\Delta_{R|K}^n(R_L \otimes_L R_L)} \cong L\otimes_K P^n_{R|K}, \]
    where $P^n_{R|K}$ denotes the module of principal parts. Under the hypotheses, $P^n_{R|K}$ is finitely presented, so by Hom and flat base change,
    \[ D^n_{R_L|L} \cong \Hom_{R_L}(P^n_{R_L|L},R_L) \cong R_L \otimes_R \Hom_{R}(P^n_{R|K},R_K) \cong L\otimes_K \Hom_R(P^n_{R|K},R),\]
    as desired.
\end{proof}

    While our main results will be in characteristic zero, we will also utilize differential operators for rings of positive characteristic. We state the results we use below in a form tailored to our situation rather than in their most general forms.

    \begin{proposition}\label{prop:charpdiff} Let $T$ be a finitely generated $\Z$-algebra.
    \begin{enumerate}
        \item Let $T$ be a $\Z$-algebra. Then reduction modulo $p$ induces a ring homomorphism
        \[ \F_p \otimes_{\Z} D_{T|\Z} \to D_{T/pT| \F_p} \]
        compatible with the order filtration that restricts to an isomorphism on the order zero piece.
        \item There is a containment $D^{p-1}_{T/pT| \F_p} \subseteq \End_{(T/pT)^p}(T/pT)$ as subsets of $\End_{\F_p}(T/pT)$.
    \end{enumerate}
    \end{proposition}

\begin{proof}
    The first statement follows from the definitions. For the second, see \cite[Proposition~5.6]{BJNB}.
\end{proof}

We will utilize use the following lemma in the next result.

\begin{lemma}[{\cite[Lemma~4.4]{AHJNTW2}}]\label{lem44}
Let $R$ be a $K$-algebra and $f\in R$. For every $t\in \Z$, in $D_{R_f|K}$ we have
\[ \delta f^t = \sum_{i=0}^{\ord(\delta)} \binom{t}{i} f^{t-i} [\delta]_f^i \qquad \text{and} \qquad f^t \delta  = \sum_{i=0}^{\ord(\delta)} (-1)^i \binom{t}{i} [\delta]_f^i f^{t-i}. \]
\end{lemma}
    
    \begin{lemma}\label{lemma:rearragnement}
    Let $R$ be a finitely generated algebra over a field $K$, and $f\in R$. Fix differential operators $\delta_1,\dots,\delta_t\in D_{R|K}$ and integers $a_1,\dots,a_t$. Then in $D_{R_f|K}$ we have equalities:
    \[ \delta_1 f^{a_1} \delta_2 f^{a_2} \cdots \delta_t f^{a_t} = \sum_{i_t=0}^{\ord{\delta_t}} \cdots \sum_{i_1=0}^{\ord{\delta_1}}  A_t \cdots A_1 f^{a_1+\cdots+ a_t - i_1 - \cdots- i_t} [\delta_1]_f^{i_1} \cdots [\delta_t]_f^{i_t}, \]
    where \[A_j = \binom{\sum_{k=j}^t a_k - \sum_{k={j+1}}^t i_k}{i_j};\]
    and
     \[  f^{a_1} \delta_1 f^{a_2}  \delta_2 \cdots f^{a_t} \delta_t  = \sum_{i_t=0}^{\ord{\delta_t}} \cdots \sum_{i_1=0}^{\ord{\delta_1}} (-1)^{i_1+\cdots +i_t} A_t \cdots A_1 [\delta_1]_f^{i_1} \cdots [\delta_t]_f^{i_t} f^{a_1+\cdots+ a_t - i_1 - \cdots- i_t}. \]
    \end{lemma}
    \begin{proof}
     We apply Lemma~\ref{lem44} to prove the first equality by induction on~$t$, with $t=1$ an immediate consequence of the aforementioned equality.
    Applying Lemma~\ref{lem44} to the last pair of products, we obtain
    \[ \delta_1 f^{a_1} \delta_2 f^{a_2} \cdots \delta_t f^{a_t} = \sum_{i_t=0}^{\ord{\delta_t}} \binom{a_t}{i_t} \delta_1 f^{a_1} \delta_2 f^{a_2} \cdots \delta_{t-1} f^{a_{t-1}+a_t-i_t} [\delta_t]_f^{i_t}.\]
    By the induction hypothesis, we can rewrite
    \[ \delta_1 f^{a_1} \delta_2 f^{a_2} \cdots \delta_{t-1} f^{a_{t-1}+a_t-i_t} = \sum_{i_{t-1}=0}^{\ord{\delta_{t-1}}}\cdots \sum_{i_{1}=0}^{\ord{\delta_{1}}} A_{t-1} \cdots A_1 f^{a_1+\cdots+ a_t - i_1 - \cdots- i_t} [\delta_{1}]_f^{i_{1}} \cdots  [\delta_{t-1}]_f^{i_{t-1}},\]
    where the constants $A_{t-1},\dots,A_1$ agree with those in the statement. Substituting in, the desired equality is seen to hold. 
    
    For the second equality, we proceed similarly.
    \end{proof}

\subsection{Filtered rings and modules}

Our notion of dimension for modules over rings of differential operators will be based on filtrations.

\begin{definition}\label{def:filt1}
    Let $A$ be a ring (not necessarily commutative) and let $\cFb = F^1, F^2,\ldots$ be a sequence of $A$-modules. We say that $(A,\cFb)$ or $R$ is a  {\textit{filtered ring}} or that $\cFb$ is a  {\textit{filtration on $A$}} if 
    \begin{itemize}
        \item $\cFb$ is ascending: $F^1\subseteq F^2 \subseteq \cdots$,
        \item $\cFb$ is exhaustive: $\bigcup_i F^i = A$,
        \item $\cFb$ has the multiplicative property: $F^iF^j \subseteq F^{i+j}$ for any $i,j\geq 0$.
    \end{itemize}
    When $R$ is an algebra over some field $K$, we say that $(A,\cFb)$ is a  {\textit{filtered $K$-algebra}} or that $\cFb$ is a \textit{$K$-filtration on $A$} if $\cFb$ is a filtration on $A$ and $\cFb$ is a sequence of finite dimensional $K$-vector spaces with $F^0 =K$.
\end{definition}

Note that, with this convention, the order filtration is a filtration on $D_R$ but not a $K$-filtration on $A$.

\begin{definition}\label{def:filt2}
    Let $(R, \cFb)$ be a filtered ring. For a left (resp.~right) $A$-module $M$ and a sequence of abelian groups $\cGb = G^1, G^2,\ldots$, we say that $(M,\cGb)$ or $M$ is a  {\textit{filtered left (or right) $A$-module}} or that $\cGb$ is a  {\textit{left (or right) filtration on $M$ compatible with $\cFb$}} if 
    \begin{itemize}
        \item $\cGb$ is ascending: $G^1\subseteq G^2 \subseteq \cdots$,
        \item $\cGb$ is exhaustive: $\bigcup_i G^i = M$,
        \item $\cGb$ has the multiplicative property $F^iG^j \subseteq G^{i+j}$ (or $G^iF^j \subseteq G^{i+j}$) for any indices $i$ and $j$.
    \end{itemize}
\end{definition}

As in the definitions above, we will often use caligraphic font the name of a filtration and the corresponding roman letters for the elements in the sequence.

\begin{remark}
    Let $A$ be a ring and $M$ an $A$-module. For any filtration $\cFb$ of $A$, the module $M$ admits a filtration $\cGb$ compatible with $\cFb$.   Thus, the condition that $(M,\cGb)$ is an $(A,\cFb)$-module is not a restriction on $M$, but rather an enhancement of $M$.

    Moreover, if $\cFb$ is a $K$-filtration on $A$, and $M$ is countably generated, then $M$ admits a filtration $\cGb$ compatible with $\cFb$ by finite-dimensional $K$-vectorspaces: if $M=\sum_{i=0}^\infty A m_i$, one can take $G^i = \sum_{j=0}^i F^{i-j} m_j$.
\end{remark}

Since it is possible to impose multiple filtrations on a given ring or module, it is useful to compare two given filtrations for such objects. We will utilize the following notions to make such comparisons, which come from \cite{AHJNTW2}:
\begin{definition}
    Let $\cFb$ and $\cGb$ be two filtrations on a ring or module.
       We say $\cGb$ \textit{linearly dominates} $\cFb$ if there is some fixed $C>0$ such that $F^i\subseteq G^{Ci}$ for all indices $i\in \N$. If $\cFb$ and $\cGb$ linearly dominate each other, we say that they are \textit{linearly equivalent}.
\end{definition}
\begin{definition}
    Let $(A,\cFb)$ be a ring and let $M$ be an $R$-bimodule. If $\cGb$ is both a left filtration and a right filtration on $M$ compatible with $\cFb$, we say that $(M,\cGb)$ or $M$ is a  {\textit{filtered $A$-bimodule}}.
\end{definition}

\begin{remark}\label{rem:box}
Recall that if $A$ is a $K$-algebra, and $M$ is an $A$-bimodule, then $M$ can be viewed as an $A\otimes_K A^{\op}$-module by the rule
\[ (\sum_i r_i \otimes s_i) \cdot m = \sum_i r_i \cdot m \cdot s_i.\]
If $(A,\cFb)$ is a filtered $K$-algebra, then $A\otimes_K A^{\op}$ is a filtered $K$-algebra via the filtration  $\cF_{\boxtimes}^\bullet$ with $i$th term given by $F^i_{\boxtimes}:= F^i \otimes_K F^i$.
If  $(M,\cGb)$ is a filtered $(A,\cFb)$-bimodule, then $(M,\cGb)$ can equivalently be viewed as a filtered $(A\otimes_K A^{\op},\cF_{\boxtimes}^\bullet)$-module.
\end{remark}

\begin{definition}\label{def:dim}
    Let $(M,\cGb)$ be a filtered module over some filtered $K$-algebra. We denote the  {\textit{dimension of the filtration}} by $\Dim(\cGb)$, which we define by the following equivalent formulae: 
    \begin{align*}
        \Dim(\cGb)  &=  \inf\left\{t\in\RR_{\geq0} \mid \lim_{n\to\infty}{\frac{\dim_K{G^n}}{n^t}} = 0\right\}\\
                    &=  \inf\left\{t\in\RR_{\geq0} \mid \dim_K{G^n} \leq n^t\ \forall i \gg 0 \right\}\\
                    &=  \limsup_{n\to\infty}{\frac{\log(\dim_K(G^n))}{\log(n)}}.
    \end{align*}
    When the dimension $\Dim(\cGb) = d$ is finite we denote the  {\textit{multiplicity of the filtration}} by $\e(\cGb)$, defined as
        \begin{align*}
            \e(\cGb)    &=  \limsup_{n\to\infty}{\frac{\dim_K(G^n)}{n^d}} = \limsup_{n\to\infty}{\frac{d!\dim_K(G^n)}{n^d}}.
        \end{align*}
\end{definition}
\begin{remark}
If $\cFb$ and $\cGb$ are two linearly equivalent $K$-filtrations on a module $M$, then $\Dim(\cFb) = \Dim(\cGb)$, and moreover, if this common dimension is finite, then $0<\e(\cFb) < \infty$ if and only if $0<e(\cGb)<\infty$ \cite[Corollary~2.6]{AHJNTW2}.

However, for filtrations that are not linearly equivalent, these statements may fail.
  %  For a field $K$, let $R = K[x]$. Let $M$ be a cyclic $R$-module generated by $m$. Define the sequences $\cBb$, $\cGb$, and $\mathcal{H}^\bullet$ by $B^i = K \cdot \{x^a \partial_x^b \mid a+b \leq i\}$, $G^i = B^i\cdot \{m\}$, and $H^i = G^{(i^2)}$ respectively. Then $\cBb$ is a filtration on $D_{R|K}$, and both $\cGb$ and $\mathcal{H}^\bullet$ are filtrations on $M$ compatible with $(D_{R|K}, \cBb)$. Computing the vector space dimensions of the filtration pieces, we see that $\dim_k(G^n) = \frac{n(n+1)}{2}$ and $\dim_k(H^n) = \frac{n^2(n^2+1)}{2}$. With this we see that $\Dim(\cGb) = 2$ while $\Dim(\mathcal{H}^\bullet) = 4$.
\end{remark}

\begin{remark}\label{remark-double-filtration}
If $(A,\cFb)$ is a filtered $K$-algebra and $\Dim(\cFb), \ee(\cFb)$ are both finite, then
\[ \Dim(\cFb_{\boxtimes}) = 2 \Dim(\cFb) \quad \text{and} \quad \ee(\cFb_{\boxtimes}) = \ee(\cFb)^2.\]
\end{remark}

\begin{definition}\label{def:lin simple} Let $A$ be a (not necessarily commutative) algebra over a field $K$. We say that $A$ is \emph{simple} if $R$ has no nonzero proper two-sided ideals.
If $(A,\cFb)$ is a filtered $K$-algebra, we say that $(A,\cFb)$ is \emph{linearly simple} if there is a constant $C$ such that $r\in F^t\smallsetminus 0$ implies $1\in F^{Ct} r F^{Ct}$.
\end{definition}

\begin{remark}
It will be useful later to recast the notions of simplicity and linear simplicity in terms of the $A\otimes_K A^{\op}$-module structure on $A$ corresponding to its natural bimodule structure. We have that
\begin{itemize}
\item $A$ is simple if and only if for any nonzero $r\in A$, there is some $\alpha\in A\otimes_K A^{\op}$ with $\alpha(r)=1$, and
\item $(A,\cFb)$ is linearly simple if and only if there is a constant $C$ such that for any $r\in F^t$, there is some $\alpha\in F^{Ct}_{\boxtimes}$ with $\alpha(r) =1$.
\end{itemize}
\end{remark}

\subsection{Filtrations on rings of differential operators}

In the sequel, we will need to consider various constructions of filtrations on rings of differential operators.

\subsubsection{Generator filtrations}

\begin{definition}\label{def:gen filt}
    Let $R$ be a finitely generated algebra over a field $K$. Suppose that the ring of differential operators $D_{R|K}$ is a finitely generated $K$-algebra. For a finite generating set $S=\{\delta_1,\dots,\delta_t\}$ of $D_{R|K}$, set $V^0=K\cdot 1$, $V^1=V^0 + K \cdot S$, and $V^i = K \cdot (V^1)^i$ for all $i>1$. We call this the \textit{generator filtration} on $D_{R|K}$ associated to the generating set $S$.
\end{definition}

The following is standard; cf.~\cite[Proposition~2.13]{AHJNTW2}:
\begin{lemma}\label{fg-filt-comapre} Under the hypotheses of Definition~\ref{def:gen filt}, if $S$ a finite generating set for $D_{R|K}$ with corresponding filtration $\cFb$, and $\cGb$ is any other filtration, then $\cFb$ is linearly dominated by $\cGb$. In particular, if
$S'$ is any other finite generating set for $D_{R|K}$ and $\cGb$ is the associated generator filtration, then $\cFb$ and $\cGb$ are linearly equivalent.
\end{lemma}

\subsubsection{Bernstein filtrations}\label{sssec:BF}

The condition that $D_{R|K}$ is a finitely generated $K$-algebra is restrictive and not well-understood in general. The Bernstein filtrations defined in \cite{AHJNTW2} are defined for rings of differential operators over graded rings (without finite generation hypotheses on the ring of differential operators), and satisfy similar properties to generator filtrations.

\begin{definition} [{\cite[Definition~4.15]{AHJNTW2}}]\label{def:Ber filt} Let $R$ be a finitely generated $\NN$-graded algebra over a field $K$ with $R_0=K$. For an integer $w$ such that $R$ is generated as a $K$-algebra by elements of degree strictly less than $w$, the \textit{Bernstein filtration} with slope $w$ on $D_{R|K}$ is the $K$-algebra filtration $\cBb$ given by 
\[B^i = K \cdot \{ \delta \ \mathrm{homogeneous} \ | \ \deg(\delta) + w \, \mathrm{ord}(\delta) \leq i\}.\]
\end{definition}

We note that the classical Bernstein filtration on a standard graded polynomial ring is the Bernstein filtration with slope $2$.

\begin{remark}\label{rem:intersection}
    For a Bernstein filtration $\cBb$, any element of $f\in D_{R}^0 \cap B^i$ must have order zero, since any such $f$ would be an element of $D^0 \cong R$ and thus has no partial derivative factors in any of its terms. Thus
    \begin{align*}
        D_R^0 \cap B^i    &=  \{f\in R\ | \ \deg(f) +w\ord(f)\leq i\}\\  
                        &=  \{f\in R\ | \ \deg(f)\leq i\}\\ 
                        &=  [R]_{\leq i}
    \end{align*}
\end{remark}

We also omit the slope from the notation of the Bernstein filtration; for our purposes, this is justified by the following Lemma.

\begin{lemma}[{{\cite[Proposition~4.20]{AHJNTW2}}}]\label{lem:berfilt} In the setting of Definition~\ref{def:Ber filt}, for $w$ and $w'$ as in the definition, then the associated Bernstein filtrations are linearly equivalent.
\end{lemma}

We also have the following.

\begin{lemma}[{\cite[Proposition~4.19]{AHJNTW2}}]\label{lem:lindom}
In the setting of Definition~\ref{def:Ber filt}, the Bernstein filtration $\cBb$ is linearly dominated by the order filtration.
\end{lemma}

\subsubsection{Localized filtrations}

\begin{definition}\label{def:loc filt}
Let $R$ be a finitely generated algebra over a field $K$.
Let $\cFb$ be a filtration on $D_{R|K}$. Let $f\in R$. We define the \textit{localized filtration at $f$} of $\cFb$ to be the smallest filtration $\cGb$ on $D_{R_f|K}$ such that $F^n\subset G^n$ for all $n$ and $f^{-1}\in G^1$; concretely
\[ G^n = \sum_{\substack{\sum a_i+\sum b_i \leq n \\ a_i, b_i \geq 0}} f^{-{a_1}} F^{b_1} \cdots f^{-{a_m}} F^{b_m} .\]
\end{definition}

\begin{remark} If $D_{R|K}$ is finitely generated and $\cF^{\bullet}$ is a generator filtration of $D_{R|K}$, then the corresponding localized filtration on $D_{R_f|K}$ is just the generator filtration on $D_{R_f|K}$ obtained by adjoining $f^{-1}$ to the generating set.
\end{remark}

\begin{proof}
Let $R$ be a ring such that $D_{R|K} = K[S]$ is a finitely generated $K$-algebra, and let $f\in R$ be nonzero. Let $\cFb$ be the generator filtration of $D_{R_f|K}$ for the generating set $\{1/f\} \cup S$, and let $\cGb$ be the generator filtration of $D_{R|K}$ for the generating set $S$. Finally let $\mathcal{H}^\bullet$ be the localization of $\cGb$ at $f$. Our goal is to show that $F^i = H^i$.

Note that for any fixed $i\in \N$, a $K$-basis for $F^i$ is given by the monomial expression in $\{1/f\}\cup S$ with at most $i$ factors, which is also a basis for $H^i$. More explicitly, observe that
\begin{align*}
    F^i &=  K\cdot\left\{\frac{1}{f} \cup S\right\}_i\\
        &=  K\cdot\left\{f^{-a_1}s_1f^{-a_2}s_2\cdots f^{-a_m}s_m \mid s_i\in K[S], \sum a_i + \sum \deg(s_i) \leq n\right\}\\
        &=  \sum_{\substack{\sum a_i+\sum b_i \leq n \\ a_i, b_i \geq 0}} f^{-{a_1}} F^{b_1}\cdots f^{-a_m}F^{b_m} = H^i. \qedhere
\end{align*}
\end{proof}

\begin{lemma}\label{lem:loc-ber-const}
    Let $R$ be an $\N$-graded $K$-algebra over a field $K$ and $f\in R$. Let $\cBb$ be a Bernstein filtration on $D_{R|K}$, and $\cGb$ the localized filtration of $\cBb$ at $f$. Then there exists a constant $C$ such that $G^n \subseteq f^{-n} B^{Cn}$ for all $n$. 
\end{lemma}
\begin{proof}
    Take $\delta\in G^n$. Fix $d$ such that $f\in B^d$, and, by Lemma~\ref{lem:lindom}, a constant $e$ such that $B^k \subseteq D^{ek}$ for all $k$.
    
    By definition, we can rewrite $\delta$ as a sum of elements of the form
    \[ \eta=  \delta_1 f^{-a_1} \delta_2 f^{-a_2} \cdots  \delta_t f^{-a_t}\]
    with $\delta_j \in B^{b_j}$, where $\sum a_j + \sum b_j \leq n$. By Lemma~\ref{lemma:rearragnement}, we can rewrite $\eta$ as a sum of elements of the form
    \[ A f^{-\sum a_j - \sum i_j} [\delta_1]^{i_1}_{f} \cdots [\delta_t]^{i_t}_{f}\]
    for some constant $\alpha \in \Z$ and $0\leq i_j \leq \ord(\delta_i)$. We have $\ord(\delta_i)\leq e b_j$, and since ${[B^k,f] \subseteq B^{k+d-1}}$, we also have $[\delta_j]^{i_j}_{f}\subseteq B^{b_j+ e(d-1) b_j}$. 

    We then have 
    \[\begin{aligned} \alpha &f^{-\sum a_j - \sum i_j} [\delta_1]^{i_1}_{f} \cdots [\delta_t]^{i_t}_{f} = \alpha f^{-n} f^{n-\sum i_j} [\delta_1]^{i_1}_{f} \cdots [\delta_t]^{i_t}_{f} \\
    &\subseteq f^{-n} B^{nd} B^{b_1+ e(d-1) b_1} \cdots B^{b_t+ e(d-1) b_t} \subseteq f^{-n} B^{nd+ (e(d-1) + 1)(\sum b_j)} \subseteq f^{-n} B^{n(d+ed - e + 1)}, \end{aligned}\]
    so the constant $C=ed+d-e+1$ works.
\end{proof}

\subsubsection{Conditions for linear simplicity}

 \begin{lemma}\label{reduce to R bar}
        Let $R$ be a finitely generated algebra over a field $K$. Let $\cFb$ be a filtration of $D_{R|K}$ that is linearly dominated by the order filtration. Suppose that there is some $C>0$ such that for all nonzero $f\in D^0 \cap F^n$ one has $1\in F^{Cn} f F^{Cn}$. Then $(D_{R|K},\cFb)$ is  linearly simple. 
    \end{lemma}
    \begin{proof}
    Let $b\in \NN$ be such that $\cF^n \subseteq D^{bn}_{R|K}$.
     Fix a finite generating set $S$ of $R$ as a $K$-algebra. Since $\mathrm{span}_K(S)$ is finite dimensional, there is some $a\in \NN$ such that $S\subseteq F^a$. Let $\delta\in F^n \smallsetminus 0 \subseteq D_{R|K}^{bn} \smallsetminus 0$. By Lemma~\ref{lem:bracket-order} above, there exists a sequence $s_1,\dots,s_{m}\in S$ with $m\leq bn$ such that $[\delta]^{1,\dots,1}_{s_1,\dots,s_m} \in D^0 \smallsetminus 0$. Since each $s_i\in F^a$, we have \[f=[\delta]^{1,\dots,1}_{s_1,\dots,s_m} \in F^{am} \delta F^{am} \subseteq F^{abn} \delta F^{abn} \subseteq F^{(2ab+1)n}\cap D^0 \smallsetminus 0.\] Thus, by the hypothesis,
     \[ 1\in F^{(2ab+1)cn} f F^{(2ab+1)cn} \subseteq F^{(2ab+1)cn} F^{abn} \delta F^{abn} F^{(2ab+1)cn} \subseteq F^{(2abc+ab+c)n} \delta F^{(2abc+ab+c)n}.\]   
     Thus $(D_{R|K},\cFb)$ is linearly simple with constant $2abc+ab+c$.
    \end{proof}
      
    The following proposition allows us to compare the notions of linear simplicity for different filtrations discussed above.
    
    \begin{proposition}\label{prop:simple-filtrations}
    Let $R$ be a finitely generated algebra over a field $K$.
    \begin{enumerate}
        \item If $D_{R|K}$ is finitely generated, then linear simplicity is equivalent for any two generator filtrations of $D_{R|K}$;
        \item If $R$ is $\NN$-graded with $R_0=K$, then linear simplicity is equivalent for any two Bernstein filtrations of $D_{R|K}$.
        \item If $R$ is $\NN$-graded with $R_0=K$ and $D_{R|K}$ is finitely generated, then linear simplicity for any generator filtration of $D_{R|K}$ implies linear simplicity for any Bernstein filtration of $D_{R|K}$.
        \item If $R$ is $\NN$-graded with $R_0=K$, $f\in R$, and $D_{R_f|K}$ is finitely generated, then linear simplicity for any generator filtration of $D_{R_f|K}$ implies linear simplicity for any localized (at $f$) Bernstein filtration on $D_{R_f|K}$.
    \end{enumerate}
     \end{proposition}
    \begin{proof} Cases (1) and (2) follow from Lemmas~\ref{fg-filt-comapre} and \ref{lem:berfilt} and the fact that linear simplicity is equivalent for two linearly equivalent filtrations \cite[Proposition~4.21]{AHJNTW2}.
    
    For (3), note that it suffices to choose any generating set for $D_{R|K}$ by (1). Pick a generating set that contains a homogeneous generating set for $R$ as a $K$-algebra, and let $\cFb$ be the corresponding filtration. Let $B^\bullet$ be any Bernstein filtration. Let $C$ be a constant of linear simplicity for $\cFb$, and $a$ be such that $F^m \subseteq B^{am}$ for all $m$, which exists by \cite[Proposition~4.19]{AHJNTW2}. Let $f\in B^n \cap D^0_{R|K}$ be nonzero. As $B^n \cap D^0_{R|K}\cong [R]_{\leq n}$, we have $f\in F^n$. Then \[ 1\subseteq F^{Cn} f  F^{Cn} \subseteq B^{Can} f B^{Can}.\] By Proposition~\ref{reduce to R bar}, we conclude that $B^\bullet$ is linearly simple.

For (4), let $\cBb$ be a Bernstein filtration on $D_R$, and $\cGb$ be the localized filtration of $\cBb$ at $f$ on $D_{R_f|K}$. 

Let $\cFb$ be a generator filtration for $D_{R_f|K}$. By Lemma~\ref{fg-filt-comapre}, without loss of generality, we can assume that a homogeneous algebra generating set for $R$ and $f^{-1}$ are contained in $F^1$. Suppose that $\cFb$ is linearly simple with constant $C_0$. By Lemma~\ref{lem:loc-ber-const}, there is a constant $C_1$ such that $G^n \subseteq f^{-n} B^{C_1 n}$.  Furthermore, since any generator filtration is linearly dominated by any other filtration, there is a constant $C_2$ such that $F^n \subseteq G^{C_2 n}$ for all $n$.

By Lemma~\ref{reduce to R bar}, to show that $\cGb$ is linearly simple, it suffices to show that there is a constant $C$ such that for any $h\in D_{R_f|K}^0 \cap G^n$, we have $1\in G^{Cn} h G^{Cn}$.

For $h\in D_{R_f|K}^0 \cap G^n$, we have $h\in f^{-n} B^{C_1 n} \cap D_{R_f|K}^0$. Writing $h=f^{-n} \delta$ with $\delta\in B^{C_1 n}$, we have $\delta=f^n h \in B^{C_1 n} \cap D^0_{R_f|K}$. Since $\delta \in B^{C_1 n} \cap D^0_{R_f|K} \subseteq D_R \cap D^0_{R_f|K} = D^0_{R|K}$, we have $\delta \in B^{C_1 n} \cap D^0_{R|K} \subseteq [R]_{\leq C_1 n}$, using Remark~\ref{rem:intersection}. From the assumption that $F^1$ contains a homogeneous algebra generating set for $R$, it follows that $[R]_{\leq C_1 n}\subseteq F^n$, and thus $h=f^{-n}\delta\in F^{(C_1+1)n}$.

Thus, we obtain
\[ 1\in F^{C_0 (C_1+1)n} h F^{C_0 (C_1+1)n} \subseteq G^{C_0 (C_1+1) C_2n} h G^{C_0 (C_1+1)C_2n}.\]
It follows that $\cGb$ is linearly simple with constant $C=C_0 (C_1+1)C_2$.     \end{proof}

    \subsection{Bernstein's inequality and linear simplicity}\label{ssec:BI}   Recall that Bernstein's inequality for a polynomial ring $R$ over a field $K$ of characteristic zero states that the dimension of any nonzero $D_R$-module is at least the dimension of $R$. In this setting, the dimension of a $D_R$-module can be characterized either by the dimension of the singular support (i.e., the dimension of an associated graded module of $M$ with respect to a good filtration compatible with the order filtration) or the dimension of a good filtration compatible with the Bernstein filtration, which is the minimal dimension of any filtration compatible with the Bernstein filtration. The last characterization of dimension is the one we work with in the results below.

   The connection between linear simplicity and Bernstein's inequality is given by the following result of Bavula:

\begin{theorem}[{\cite[Theorem~3.1]{Bav09}, \cite[Theorem~3.4]{AHJNTW2}}]
Let $(A,\cFb)$ be a filtered $K$-algebra. If $(A,\cFb)$ is linearly simple, then for any nonzero $(A,\cFb)$-module $(M,\cGb)$, we have
\[ \Dim(\cGb) \geq \frac12 \Dim(\cFb). \]
\end{theorem}

To connect this with the classical Bernstein's inequality and some of its applications, we recall the following definition from \cite{AHJNTW2}:

    \begin{definition}
        Let $R$ be a finitely generated $\N$-graded algebra over a field $K$ with ${R_0=K}$. We say that $R$ is a \textit{Bernstein algebra} if, for some Bernstein filtration $\cBb$ of $D_R$ we have
        \begin{itemize}
            \item $(D_R,\cBb)$ is linearly simple,
            \item $\Dim(D_R)=2 \dim(R)$, and
            \item $0<\e(D_R)<\infty$.
        \end{itemize}
    \end{definition}

    We note that since every Bernstein filtration on $D_R$ is linearly equivalent, each of the conditions above is independent of the choice of a Bernstein filtration.

   A large class of Bernstein algebras is given by rings of invariants of polynomial rings over rings of characteristic zero by linear actions of finite groups \cite[Theorem~5.5]{AHJNTW2}; in particular polynomial rings are Bernstein algebras.

   Towards a criterion for showing more rings are Bernstein algebras, we recall the graded version of the notion of \emph{differential signature} from \cite{BJNB}.

    \begin{definition}\label{def:diffsig}
        Let $K$ be a field and $R$ be a $K$-algebra finitely generated $\N$-graded with $R_0=K$ and $\m= R_{>0}$. The differential signature of $R$ is
        \[ \mathrm{s}^{\mathrm{diff}}_K(R) \colonequals \limsup_{n\to \infty} \frac{\dim(R)!}{n^{\dim(R)}}  \dim_K\left(\frac{R}{\m^{\langle n \rangle_K}}\right),\]
        where $\m^{\langle n \rangle_K} = \{ f\in R \ | \ D^{n-1}_{R|K} \cdot f \subseteq \m \}$.
    \end{definition}

\begin{proposition}\label{prop:posdiff}
Let $R$ be a finitely generated $\N$-graded algebra over a field $K$ with ${R_0=K}$. Let $\cBb$ be a Bernstein filtration on $D_R$. If $(D_R,\cBb)$ is linearly simple and
${\mathrm{s}^{\mathrm{diff}}_K(R)>0}$, then $R$ is a Bernstein algebra.

In particular, if $(D_R,\cBb)$ is linearly simple and $R$ is a direct summand of a polynomial ring, then $R$ is a Bernstein algebra.
\end{proposition}
\begin{proof}
The first statement is \cite[Theorem~4.32]{AHJNTW2}. The second follows from \cite[Theorem~6.15]{BJNB}.
\end{proof}

The upshot of this condition is the following:

\begin{theorem}[{\cite[Theorem~4.32, Corollary~4.13, Proposition~4.8]{AHJNTW2}}]
If $R$ is a Bernstein algebra over a field $K$ and $\cBb$ is a Bernstein filtration on $D_R$ then 
\begin{itemize}
    \item For every nonzero $(D_R,\cBb)$-module $(M,\cGb)$, we have $\Dim(\cGb)\geq \dim(R)$.
    \item The modules $R$, $R_f$ for $f\in R$ nonzero, and $H_I^i(R)$ for $i\in \N$ and $I$ an ideal of $R$, all have finite length as $D_R$-modules, and finitely many associated primes as $R$-modules.
    \item If $K$ has characteristic zero, then every nonzero element $f\in R$ admits a nonzero Bernstein-Sato functional equation.
\end{itemize}
    \end{theorem}
    
    \begin{remark}\label{rem:BIintro}
Let $\cBb$ be a Bernstein filtration. Set 
\[ \Dim(M) = \inf\{ \Dim(\cFb) \ | \ \cFb \ \text{is a filtration on $M$ compatible with $\cBb$}\}.\]
If $R$ is a Bernstein algebra, and $M$ is a nonzero $D_R$-module, we have $\Dim(M)\geq \dim(R)$. This is the precise version of Bernstein's inequality given by Theorem~\ref{thm1.1} in the intro.
\end{remark}

    We recap the most straightforward version of the connection between Bernstein's inequality and linear simplicity for easy reference.

    \begin{corollary}\label{cor:dir}
     Let $R$ be a finitely generated $\N$-graded algebra over a field $K$ over characteristic zero with ${R_0=K}$.

 If $R$ is a direct summand of a polynomial ring, and there is a Bernstein filtration $\cBb$ for $D_R$ that is linearly simple, then Bernstein's inequality holds in the sense that for any nonzero $D_R$-module $M$, the dimension of any $\cBb$-compatible filtration on $M$ is at least $\dim(R)$.
    \end{corollary}

   \section{Sandwich Bernstein-Sato polynomials}
   
   \subsection{Definition}
Let $R$ be $K$-algebra, where $K$ is a field of characteristic zero. For any $f\in D^0_{R|K}\cong R$, we say that 
 \begin{equation}\label{eq:SBSdef1} \sum_i \alpha_i(s)f^{s+1}\beta_i(s) = b(s) f^s \end{equation}
 is a  {\textit{Sandwich-Bernstein-Sato equation}} for $f$ (or  {\textit{SBS-equation}} for $f$) if $\alpha_i(s),{\beta_i(s)\in D_{R|K}[s]}$,  $b(s) \in K[s]$, and the equation
   \begin{equation}\label{eq:SBSdef2} \sum_i \alpha_i(n)f^{n+1}\beta_i(n) = b(n) f^n \end{equation}
   holds in $D_{R|K}$ for all $n\in \NN$.

   Note that in \eqref{eq:SBSdef2}, as this takes place in $D_{R|K}$, the multiplication is composition rather than the action of $D_{R|K}$ on $R$ used in the usual notion of Bernstein-Sato polynomial; see Remark~\ref{rem:standardBS} below.

 \begin{remark}
We will denote elements of $D_{R|K}[s]$ (and $D_{R_f|K}[s]$) either by a single letter like $\alpha$ or as an expression of $s$, e.g., $\alpha(s)$. In particular, we will use the notation $\alpha(s)$ when we also want to consider the evaluation $\alpha(t)\in D_{R|K}$.
 \end{remark}

\begin{proposition}\label{SBS polynomials form an ideal}
For a given $f\in D^0_{R|K}$, the polynomials $b(s)\in K[s]$ that satisfy a SBS-equation form an ideal.
\end{proposition}

\begin{proof}
Suppose $b(s)$ and $c(s)$ are polynomials satisfying two SBS-equations for some ${f\in D^0_{R|K}}$. Explicitly, say
    $$\sum_i \alpha_i(s)f^{s+1}\beta_i(s) = b(s) f^s,$$
    $$\sum_j \gamma_j(s)f^{s+1}\delta_j(s) = c(s) f^s,$$
for some $\alpha_i(s),\beta_i(s),\gamma_j(s),\delta_j(s) \in D_{R|K}[s]$. Then $b(s) + c(s)$ satisfies the SBS-equation
    $$\sum_i \alpha_i(s)f^{s+1}\beta_i(s) +  \sum_j \gamma_j(s)f^{s+1}\delta_j(s)= (b(s)+c(s)) f^s.$$
Also for any $k(s) \in K[s]$, we have $\sum_i (k(s)\alpha_i(s))f^{s+1}\beta_i(s) = k(s)b(s) f^s.$
\end{proof}
Note that the ideal above in Proposition~\ref{SBS polynomials form an ideal} is principal, which gives validity to the following definition.
\begin{definition}\label{def:sbspoly}
    For any $f\in D^0_{R|K}\cong R$, we say that the  {\textit{SBS-polynomial}} (or \textit{sandwich Bernstein-Sato polynomial}) for $f$ is the monic generator for the ideal of all polynomials satisfying an SBS equation for~$f$. Our convention is that when the ideal of all polynomials satisfying an SBS equation for~$f$ is the zero ideal, the SBS polynomial of $f$ is zero. We denote the SBS polynomial of $f$ by $\SBS{f}{R}(s)$.
\end{definition}

\begin{remark}\label{rem:standardBS}
We will also discuss the notion of \emph{Bernstein-Sato polynomial} of an element in a ring, which is classical for polynomial rings in characteristic zero, cf.~\cite{BernsteinPoly}, and has been discussed for singular rings in \cite{AMHNB, BJNB, AHJNTW1}. Namely, for an algebra $R$ over a field $K$ of characteristic zero and $f\in R$ a \emph{Bernstein-Sato functional equation} for $f$ is an equation of the form
\[ \delta(s) \cdot f^{s+1} = b(s) f^s\]
where $b(s)\in K[s]$ and $\delta(s)\in D_{R|K}[s]$; by such an equation, we mean that the corresponding equations in $R_f$ hold after evaluating $s$ at $n$ for all $n\in \Z$. The \emph{Bernstein-Sato polynomial} is the minimal monic $b(s)$ appearing in such an equation, if one exists.

We will use the nonstandard notation $\BS{f}{R}(s)$ for this Bernstein-Sato polynomial.
\end{remark}

   \subsection{The bimodule $D_{R_f}[s]\fs$}
   
  The sandwich-Bernstein-Sato functional equation can also be considered as a single equality in a suitable bimodule, which we proceed to construct.
   
   Let $K$ be a field of characteristic zero, and $R$ be a finitely generated $K$-algebra. The map
   \[ \Theta: D_{R|K} \to D_{R_f|K}[s]\]
   given by
   \[ \Theta(\delta) = \sum_{i=0}^{\mathrm{ord}(\delta)} (-1)^i \binom{s}{i}  [\delta]^i_f f^{-i}\]
   is a ring homomorphism by an argument nearly identical to \cite[Proof of Lemma~4.9]{AHJNTW2}.

   This extends to a ring endomorphism of $D_{R_f|K}[s]$ by fixing $s$ and $1/f$; we denote this extension by $\Theta$ as well.

   Define $D_{R_f|K}[s]\fs$ to be the $(D_{R|K}[s],D_{R|K}[s])$-bimodule that, as a set is simply
   \[ \{ \delta \fs \ | \ \delta\in D_{R_f|K}[s]\},\]
   as an abelian group has addition
   \[ \delta \fs + \delta' \fs = (\delta + \delta')\fs,\]
   and with action
   \[ \alpha \cdot \delta \fs \cdot \beta = (\alpha \delta \Theta(\beta))\fs.\]
  
  The point of this action is as follows.  For $n\in \ZZ$, let \[\pi_n: D_{R_f|K}[s] \to D_{R_f|K} \ \text{and} \  \widetilde{\pi}_n:D_{R_f|K}[s] \fs \to D_{R_f|K}\] be the ``evaluation at $s=n$'' maps; i.e., $\pi_n(\delta(s)) = \delta(n)$ and $\widetilde{\pi}_n(\delta(s) \fs) = \delta(n) f^n$. We have
   \[  f^n \delta  = \pi_n(\Theta(\delta)) f^n\]
   in $D_{R|K}$ for all $n\in \ZZ$ by Lemma~\ref{lem44}.  Then
  \begin{align*}
      \widetilde{\pi}_n ( \alpha \cdot \delta \fs \cdot \beta) &= \widetilde{\pi}_n (\alpha \delta \Theta(\beta) \fs) = \pi_n(\alpha \delta \Theta(\beta)) f^n \\
      &= \pi_n(\alpha) \pi_n(\delta) \pi_n(\Theta(\beta)) f^n = \pi_n(\alpha) \pi_n(\delta) f^n \pi_n(\beta)\\
      &= \pi_n(\alpha) \cdot \widetilde{\pi}_n(\delta \fs) \cdot \pi_n(\beta).
  \end{align*}

  In fact, any element of $D_{R|f}[s] \fs$ is determined by its images under the maps $\widetilde{\pi}_n$.

  \begin{lemma}\label{lem:pin-determine} Let $\delta \fs\in D_{R_f|K}[s]\fs$. The following are equivalent:
  \begin{enumerate}
      \item $\delta \fs = 0$ in $D_{R_f|K}[s]\fs$.
      \item $\delta(n) f^n = 0$ in $D_{R_f|K}$ for all $n\in \Z$.
          \item $\delta(n) f^n = 0$ in $D_{R_f|K}$ for infinitely many $n\in \Z$.
  \end{enumerate}
  If $f$ is a nonzerodivisor, then these are equivalent to
   \begin{enumerate}\setcounter{enumi}{3}
     \item $\delta(n) f^n = 0$ in $D_{R|K}$ for all $n\in \Z$.
       \item $\delta(n) f^n = 0$ in $D_{R|K}$ for infinitely many $n\in \N$.
       \end{enumerate}
  \end{lemma}
  \begin{proof} Clearly (1) $\Rightarrow$ (2) $\Rightarrow$ (3). To see (3) $\Rightarrow$ (1), note that $\widetilde{\pi}_n(\delta(s) \fs) = \delta(n) f^n$ is zero in $D_{R_f|K}$ if and only if $\delta(n)$ is zero. Take a $K$-vectorspace basis $\Lambda$ for $D_{R_f|K}$, and write $\delta(s)=\sum_{i,j} k_{i,j} \lambda_{j} s^i$ with $k_{i,j}\in K$, $\lambda_{i,j}\in \Lambda$. Then \[\delta(n) = \sum_{i,j} n^i k_{i,j} \lambda_{j} = \sum_j \left(\sum_i n^i k_{i,j}\right) \lambda_j.\] If $\delta(n)=0$ for all $n\in \N$, then $\sum_{i,j} n^i k_{i,j} \lambda_{j}=0$ for all $n\in \N$, which implies $\sum_i n^i k_{i,j}=0$ for all $j$. Then $\sum_i k_{i,j} s^i=0$ for all $j$, and hence $\delta(s)=0$.

It is also clear that  (4) $\Rightarrow$ (5)  $\Rightarrow$ (3). To see (2) $\Rightarrow$ (4), note that if $f$ is a nonzerodivisor and then $\delta\in D_{R|K}$ is nonzero, then $f^m \delta\neq 0$ for all $m$, and hence $\delta$ is nonzero in ${D_{R_f|K}\cong (D_{R|K})_f}$.
  \end{proof}
  
We will focus our interest on nonzerodivisors in light of the previous proposition.   In particular, we have the following.
  
  \begin{proposition}\label{prop:formalSBS}
  Let $K$ be a field of characteristic zero, and $R$ be a finitely generated $K$-algebra. Let $f\in R$ be a nonzerodivisor and $\alpha_1(s),\dots,\alpha_t(s), \beta_1(s),\dots,\beta_t(s)\in D_{R_f|K}[s]$. Let $b(s)\in K[s]$. Then
  \[ \sum_i \alpha_i(n) \,f^{n+1} \, \beta_i(n) = b(s) f^{n}\] in $D_{R|K}$ for all $n\in \N$ if and only if
\[\tag{$\dagger$} \sum_i \alpha_i \cdot f \fs \cdot \beta_i = b(s) \fs\] in the bimodule $D_{R_f|K}[s]\fs$. In particular, the sandwich Bernstein-Sato polynomial for $f$ is the monic generator of the ideal of polynomials $b(s)$ that appear in an equation of the form~($\dagger$).
  \end{proposition}
  \begin{proof}
  The ``if'' direction follows from the discussion above and the ``only if'' direction from Lemma~\ref{lem:pin-determine}.
  \end{proof}

Applying $K(s)\otimes_{K[s]} -$, one obtains the $D_{R(s)|K(s)}\cong D_{R|K}(s)$ bimodule $D_{R_f|K}(s) \fs$. Since $D_{R_f|K}[s]\fs$ is flat over $K[s]$, the analogue of Proposition~\ref{prop:formalSBS} holds in $D_{R_f|K}(s) \fs$ as well. 

\begin{remark}
    We observe that the $(D_{R|K}[s],D_{R|K}[s])$-bimodule structure on $D_{R_f}[s]\fs$ is $K[s]$-balanced: i.e., for any $a(s)\in K[s]$ and any $\delta \fs \in D_{R_f}[s]\fs$ we have $a(s) \delta \fs = \delta \fs a(s)$. Indeed, this follows from Lemma~\ref{lem:pin-determine} and the equality $a(t) \delta f^t = \delta f^t a(t) \in D_{R_f|K}$.

    In particular, the bimodule structure is equivalent to the left-module structure over
    \[ D_{R|K}[s] \otimes_{K[s]} D_{R|K}[s]^\op \cong (D_{R|K} \otimes_{K} D_{R|K})^\op [s].\]
\end{remark}

We note that sandwich Bernstein-Sato polynomials in finitely generated $K$-algebras are invariant under extension of the base field.

\begin{lemma} Let $K\subseteq L$ be fields of characteristic zero.
    Let $R$ be a finitely generated $K$-algebra and $f\in R$ be a nonzerodivisor. Then $\SBS{f}{R}(s) = \SBS{f}{L \otimes_K R}(s)$.
\end{lemma}
\begin{proof}
By Lemma~\ref{Localizing Things}(3) we have $D^i_{L\otimes_K R|L}\cong L\otimes_K D^i_{R|K}$ for all $i$. Given an SBS equation of $f$ over $R$,
\[ \sum_i \alpha_i(s) f^{s+1} \beta_i(s) = b(s) f^s,\]
we can identify $\alpha_i(s)$, $\beta_i(s)$ with operators in on $L\otimes_K R$ and use the same equation there. This shows that
$\SBS{f}{L \otimes_K R}(s)$ divides  $\SBS{f}{R}(s)$.

On the other hand, given an SBS-equation of $f$ over $L\otimes_K R$,
\[ \sum_i \alpha_i(s) f^{s+1} \beta_i(s) = b(s) f^s,\]
we can write $\alpha_i(s) = \sum \ell_{i,j} \otimes \alpha'_{i,j}(s)$ with $\ell_{i,j}\in L$, $\alpha_{i,j}(s)\in D_{R|K}[s]$, and $b(s)=\sum_j \ell_j b_j(s)$ with $\ell_j\in L$ and $b_j(s)\in K[s]$. By taking a $K$-vectorspace basis for $K\cdot\{ \ell_{i,j}, \ell_j\}$ and collecting terms, we can rewrite our SBS-equation as an equation of the form
\[ \sum_k \ell_k \left(\sum_i \alpha''_{i,k}(s) f^{s+1} \beta''_{i,k}(s) \right) = \sum_k \ell_k b'_k(s).\]
By linear independence of $\ell_k$, this equation holds for $s=t\in \N$ if and only if 
\[\sum_i \alpha''_{i,k}(t) f^{t+1} \beta''_{i,k}(t) = b'_k(t) f^t\] holds for all $k$.
In particular, for each $k$,
\[\sum_i \alpha''_{i,k}(s) f^{s+1} \beta''_{i,k}(t) = b'_k(s) f^s\] is an SBS-equation for $f$ in $R$. Then $b(s) = \sum_k \ell_k b'_k(s) \in \SBS{f}{R}(s)\cdot L[s]$. We conclude that $\SBS{f}{L \otimes_K R}(s) = \SBS{f}{R}(s)$.
\end{proof}

\section{Bernstein inequality for bimodules and existence of sandwich Bernstein-Sato polynomials}

In this section, we will establish a version of Bernstein's inequality for bimodules and apply this to obtain sufficient conditions for existence of sandwich Bernstein-Sato polynomials. We will require a strengthening of the linear simplicity condition for a filtered algebra. In particular, it turns out that linear simplicity of the tensor product $D_{R|K} \otimes_K D_{R|K}^{\op}$ will be the appropriate condition.

We start by investigating this property.

\begin{lemma}
Let $K$ be a field, and $R$ be a $K$-algebra. Let $\cF^{\bullet}$ be a filtration on $D_{R|K}$. If there is a constant $C$ such that for any finite set of $K$-linearly independent elements $\delta_1,\dots,\delta_n\in F^t$, there is some $A= \sum_j \alpha_j \otimes \beta_j \in F^{Ct}\otimes F^{Ct}$  such that 
\begin{itemize}
    \item $\sum_j \alpha_j \delta_i \beta_j\in K$ for all $i$, and
\item $\sum_j \alpha_j \delta_i \beta_j\neq 0$ for some $i$,
\end{itemize} then $(D_{R|K} \otimes_K D_{R|K}^{\op}, \cFb_{\boxtimes})$ is linearly simple.
\end{lemma}
\begin{proof}
Let $D=\sum_i \delta_i \otimes \delta'_i \in F^t_{\boxtimes}= F^t\otimes F^t$ be nonzero; without loss of generality we can assume that $\{ \delta_i\}$ and $\{ \delta'_i\}$ are linearly independent. Take $A$ as in the statement for $\{\delta_i\}$. Then
\[
\sum_j (\alpha_j \otimes 1) D (\beta_j \otimes 1) = \sum_i \sum_j \alpha_j \delta_i \beta_j \otimes \delta'_i = \sum_i \lambda_i \otimes \delta'_i = 1 \otimes (\sum_i \lambda_i \delta'_i)
\]
for some constants $\lambda_i\in K$, not all zero. Since $\{\delta'_i\}$ are $K$-linearly independent, we have ${\delta'_0:=\sum_i \lambda_i \delta'_i\neq 0}$
 applying the hypothesis to the singleton $\{\delta'_0\}$ in $F^{Ct+t+Ct} = F^{(2C+1)t}$, there is some ${A' = \sum_k \alpha'_k \otimes \beta'_k \in F^{C(2C+1)t}\otimes F^{C(2C+1)t}}$ such that $\sum_k \alpha'_k \delta'_0 \beta'_k=1$. Then 
\[
\sum_k (1 \otimes \beta'_k) (1\otimes \delta'_0) (1 \otimes \alpha'_k) = 1\otimes (\sum_k \alpha'_k \delta'_0\beta'_k) = 1\otimes 1.\]
Combining $A$ and $A'$, we get that 
\[ \sum_{j,k} (\alpha_j \otimes 1)(1 \otimes \beta'_k) D (\beta_j \otimes 1)(1 \otimes \alpha'_k) = 1\otimes 1,\]
with $(\alpha_j \otimes 1)(1 \otimes \beta'_k)$ and $(\beta_j \otimes 1)(1 \otimes \alpha'_k)$ in $F^{(C+1)(2C+1)t}_{\boxtimes}$. Thus, $(D_{R|K} \otimes_K D_{R|K}^{\op}, \cFb_{\boxtimes})$ is linearly simple.
\end{proof}

For a particular class of $K$-algebras, we can verify the hypotheses above.

\begin{proposition}\label{prop:toricdouble}
Let $K$ be a field of characteristic zero. Let $R$ be a finitely generated $K$-subalgebra of a polynomial ring that is generated by monomials. Fix a Bernstein filtration $\cBb$ for $D_{R|K}$. If $(D_{R|K},\cBb)$ is linearly simple, then $(D_{R|K} \otimes_K D_{R|K}^{\op}, \cBb_{\boxtimes})$ is linearly simple. 
\end{proposition}
\begin{proof}
We will show that the condition of the previous lemma holds. Suppose that $D_{R|K}$ is linearly simple with constant $C'$. Assume that $R$ is generated in degree at most $w$ as a $K$-algebra, and let $a$ be such that $D_R^n\subseteq F^{an}$ for all $n$.

Let $\delta_1,\dots,\delta_n \in F^t$ be linearly independent.  Then by applying at most $at$ iterated brackets with generators of $R$, which corresponds the action of an element $A\in B^{awt}_{\boxtimes}$,  each $\delta_i$ becomes an element $r_i$ of $D^0_R \cap B^{(aw+1)t}$, not all zero. Let $x^{\alpha}= x_1^{\alpha_1}\cdots x_d^{\alpha_n}$ be a monomial that is maximal under divisibility in the union of the monomial supports of $r_1,\dots,r_n$. Observe that $\alpha_1+\cdots +\alpha_d \leq (aw+1)t$, since $D^0_R \cap B^{(aw+1)t}=[R]_{\leq (aw+1)t}$.

For each variable $x_i$ in the ambient polynomial ring $K[x_1,\dots,x_d]$, there is an Euler operator $E_i \colonequals x_i \partial_{x_i}$ that is well-defined on~$R$.

 Then, for a monomial $x^\beta$, we have that
\[\tag{$\dagger$} \left[E_1,\left[E_1-1,\cdots,\left[E_1-\alpha_1+1, \cdots , \left[E_d,\left[E_d-1,\cdots,\left[E_d-\alpha_d+1, x^\beta\right]\right] \cdots \right]\right.\right.\right. \]
 equals $0$ for any $x^\beta \neq x^\alpha$ in the support of $r_1,\dots,r_n$, and is a nonzero multiple of $x^\alpha$ if $\beta=\alpha$. Thus, $r_1,\dots,r_n$, applying these operators we obtain a multiple of $x^\alpha$, and for some $i$ we obtain a nonzero multiple; these operators correspond to some $A'\in B^{(aw+1)t}_{\boxtimes}$, and $x^{\alpha}\in B^{(aw+1)t}$.
 
 Finally, by linear simplicity, there is some element $A''\in B^{C(aw+1)t}_{\boxtimes}$ such that the corresponding sequence of operators sends these scalar multiples of $x^\alpha$ to elements of $K$, not all zero. The operators corresponding to the composition of $A,A',A''$ give an element of $B^{awt +(aw+1)t+C(aw+1)t}_{\boxtimes} = B^{((C+2)aw+(C+1))t}_{\boxtimes}$.
 
 This shows that $R$ satisfies the condition of the previous lemma.
\end{proof}

\begin{theorem}
Let $R$ be a finitely generated $\N$-graded algebra over a field $K$ of characteristic zero. Let $\cBb$ be a Bernstein filtration on $D_{R|K}$, and suppose that $\Dim(D_{R|K},\cBb)=2\dim(R)$ and $0<\e(D_{R|K},\cBb)<\infty$.
If $(D_{R|K} \otimes_K D_{R|K}^{\op}, \cBb_{\boxtimes})$ is linearly simple, then every nonzero  filtered $D_{R|K}$-bimodule has dimension at least equal to $2\dim(R)$, and when equality holds, the multiplicity is greater than zero.
\end{theorem}
\begin{proof}
We consider any filtered $D_{R|K}$-bimodule as a filtered left $(D_{R|K} \otimes_K D_{R|K}^{\op}, \cBb_{\boxtimes})$-module. Since $\Dim(D_{R|K} \otimes_K D_{R|K}^{\op}, \cBb_{\boxtimes}) = 4\dim(R)$ by Remark~\ref{remark-double-filtration}, the result follows from \cite[Theorem~3.4]{AHJNTW2}.
\end{proof}

\begin{definition}
Let $R$ be a graded algebra over a field $K$ of characteristic zero. Let $\cBb$ be a Bernstein filtration on $D_{R|K}$, and suppose that $\Dim(D_{R|K},\cBb)=2\dim(R)$, that ${0<\ee(D_{R|K},\cBb)<\infty}$
and that $(D_{R|K} \otimes_K D_{R|K}^{\op}, \cBb_{\boxtimes})$ is linearly simple.

We say that a $D_{R|K}$-bimodule $M$ is {\emph{holonomic}} if there is some filtration $(M,\cGb)$ such that $\Dim(M,\cGb)=2\dim(R)$ and ${\ee(M,\cGb)<\infty}$.
\end{definition}

\begin{corollary}\label{cor:fin-length}
In the setting of the previous definition, every holonomic bimodule has finite length.
\end{corollary}
\begin{proof}
This follows from \cite[Theorem~3.4]{AHJNTW2}, again by identifying bimodules with modules over the tensor product algebra.
\end{proof}

\begin{proposition}\label{prop:fs-bimod-holo}
Let $K$ be a field of characteristic zero and $R$ be a finitely generated $\N$-graded $K$-algebra with $R_0=K$. Suppose that $D_{R|K}$ with the Bernstein filtration is linearly simple, with dimension $2\dim(R)$,  and positive finite multiplicity. Then $D_{R_f(s)|K(s)}\fs$ is a holonomic $D_{R(s)|K(s)}$-bimodule.
\end{proposition}
\begin{proof}
We recall that $D_{R(s)|K(s)} \cong D_{R|K} \otimes_K K(s)$. The Bernstein filtration on $D_{R(s)|K(s)}$ is obtained from the Bernstein filtration on $D_{R|K}$ by base change $K(s) \otimes_K -$, so $D_{R(s)|K(s)}$ also has dimension $2\dim(R)$ and positive finite multiplicity, and is also linearly simple; thus it makes sense to consider holonomic $D_{R(s)|K(s)}$-bimodules.

Denote the Bernstein filtration on $D_{R(s)|K(s)}$ by $\cBb$. Let $d\in \NN$ be such that $f\in B^a$ and $C$ be such that $B^i \subseteq D^{Ci}_{R(s)|K(s)}$ for all $i$.

Let $\cFb$ be the filtration on $D_{R_f(s)|K(s)} \fs$ given by $F^i = B^{(Ca+1)i} \frac{1}{f^{Ci}} \fs$. We first show that this is compatible with the bimodule action of $(D_{R(s)|K(s)}, \cBb)$. Let $\delta f^{-Ci}\fs \in F^i$, so $\delta\in B^{(Ca+1)i}$. Let $\alpha\in B^{u}$ and $\beta\in B^v$; note that $\beta\in D^{Cv}_{R(s)|K(s)}$. Then, using Lemma~\ref{lem44}, we have
\begin{align*}
    \alpha \cdot \delta f^{-Ci} \fs \cdot \beta 
    &= \alpha \delta f^{-Ci} \Theta(\beta) \fs 
    = \sum_{j\leq cv} \lambda_j \alpha \delta f^{-Ci}  [\beta]^j_f f^{-j} \fs \\ 
    &= \sum_{j\leq cv} \sum_{k\leq Cv-j} \lambda_{j,k} \alpha \delta  [\beta]^{j+k}_f f^{-Ci-j-k} \fs 
    = \sum_{\ell\leq cv} \lambda_\ell \alpha \delta [\beta]^\ell_f f^{-Ci-\ell} \fs \\ & 
    \in \sum_{\ell\leq cv} B^u B^{(Ca+1)i} B^{v+a\ell} f^{-Ci-\ell} \fs \subset \sum_{\ell\leq cv} B^{u+(Ca+1)i+v+a\ell} f^{-Ci-\ell} \fs \\ 
    &\subset \sum_{\ell\leq cv} B^{u+(Ca+1)i+v+a\ell} B^{Cau+ Cav-a\ell} f^{-C(u+i+v)} \fs\\
    &\subset B^{u+(Ca+1)i+v+Cau+Cav} f^{-C(u+i+v)} \fs\\
    & \subset B^{(Ca+1)(u+i+v)} f^{-C(u+i+v)} = F^{u+i+v},
\end{align*}
as required. It is easy to see that $\cFb$ is exhaustive, and it is clear that $\cFb$ is linearly equivalent to $\cBb \fs$, and thus has dimension $2\dim(R)$ and finite multiplicity.
\end{proof}

\begin{theorem}\label{thm:LSimpliesSBS}
Let $K$ be a field of characteristic zero and $R$ be an $\N$-graded domain finitely generated  over $R_0=K$. Let $\cBb$ be a Bernstein filtration on $D_{R|K}$, and suppose that $\Dim(D_{R|K},\cBb)=2\dim(R)$ and $0<\ee(D_{R|K},\cBb)<\infty$.
If $(D_{R|K} \otimes_K D_{R|K}^{\op}, \cBb_{\boxtimes})$ is linearly simple, then every nonzero element of $R$ has a nonzero sandwich Bernstein-Sato polynomial.
\end{theorem}
\begin{proof}
By Proposition~\ref{prop:fs-bimod-holo} and Corollary~\ref{cor:fin-length}, $D_{R_f(s)|K(s)}\fs$ has finite length, and hence DCC, as a $D_{R(s)|K(s)}$-bimodule. In particular, the chain of submodules
\[ D_{R(s)|K(s)} \cdot \fs \supseteq D_{R(s)|K(s)} \cdot f \fs \supseteq 
 D_{R(s)|K(s)} \cdot f^2 \fs \supseteq \cdots \]
 stabilizes, so there is some $t$ and operators $\alpha'_i(s),\beta'_i(s)\in D_{R(s)|K(s)}$ such that
 \[ f^{t} \fs = \sum_i \alpha_i(s) f^{t+1} \fs \beta_i(s).\]
 Then (e.g., using Lemma~\ref{lem:pin-determine}), we have
 \[ \fs = \sum_i \alpha'_i(s-t) f \fs \beta'_i(s-t).\]
 We can multiply through by an element $b(s)\in K[s]$ to obtain an equation of the form
 \[ b(s) \fs = \sum_i \alpha_i(s-t) \fs \beta_i(s-t)\]
 with $\alpha_i(s-t),\beta_i(s-t)\in D_{R|K}[s]$, as required.
\end{proof}

\begin{corollary}\label{cor:toric}
    If $R$ is a finitely generated $K$-subalgebra of a polynomial ring generated by monomials, and $(D_{R|K},\cBb)$ is linearly simple for some Bernstein filtration $\cBb$, then every nonzero element of $R$ admits a nonzero sandwich Bernstein-Sato polynomial.
\end{corollary}
\begin{proof}
    Follows from Theorem~\ref{thm:LSimpliesSBS} and Proposition~\ref{prop:toricdouble}.
\end{proof}

\section{Linear simplicity from one sandwich Bernstein-Sato polynomial}

\begin{proposition}\label{prop:fastsmallmult} Let $R$ be a finitely generated algebra over a field $K$. Let $\cFb$ be a filtration on $D_{R|K}$ that is linearly dominated by the order filtration. Let $f\in R$ be a nonzerodivisor and $\cGb$ be the localized filtration of $\cFb$ at $f$ on $D_{R_f|K}$. Suppose that $(D_{R_f|K},\cGb)$ is linearly simple. Then there exist $a,b\in\N$ such that for all $n$ and all nonzero $\delta\in F^n$, we have $f^a \in F^{bn} \delta F^{bn}$. 
\end{proposition}
\begin{proof}
Fix a constant $C\in \N$ such that for all $\delta\in G^n\smallsetminus 0 $ we have $1\in G^{Cn} \delta G^{Cn}$, and, by Lemma~\ref{fg-filt-comapre}, a constant $L$ such that $F^{n} \subseteq D^{Ln}_R$ for all $n$. Take $m$ such that $f\in F^{m}$.
Let $\delta\in F^n \subset G^n$. We can write
\[ 1 = \sum_i \alpha_i \delta \beta_i\]
for some operators $\alpha_i,\beta_i\in G^{Cn}$. By Lemma~\ref{Localizing Things}(1), we can write each $\alpha_i$ and $\beta_i$ as a sum of products of elements of $D_{R|K}$ and negative powers of $f$; after expanding the products of these sums, without loss of generality, we can take each $\alpha_i$ and $\beta_i$ to be a product of elements of $D_{R|K}$ and negative powers of $f$. That is, we can write
\[ \alpha_i = \alpha_{1,i} f^{-a_{1,i}} \cdots \alpha_{t,i} f^{-a_{t,i}}\]
\[ \beta_i = f^{-b_1,i} \beta_{1,i}  \cdots f^{-b_t,i} \beta_{t,i} \]
with $\alpha_{j,i}, \beta_{j,i}\in D_{R|K}$. We can take $\alpha_{j,i}\in F^{c_{j,i}}, \beta_{j,i}\in F^{d_{j,i}}$ with $\sum_j c_{j,i} + \sum_j a_{j,i} \leq Cn$ and $\sum_j d_{j,i} + \sum_j b_{j,i} \leq Cn$ for each $i$, by definition of localized filtration. Note that ${\ord(\alpha_{j,i})\leq L c_{j,i}}$ and $\ord(\beta_{j,i})\leq L d_{j,i}$.

Observe that \[[\alpha_{j,i}]_f^{k}\in F^{c_{j,i}+km}\ \text{and} \ 
[\beta_{j,i}]_f^{k}\in F^{d_{j,i}+km}.\]

Applying Lemma~\ref{lemma:rearragnement}, for each $i$, there are some scalars $\lambda_{i,\underline{j}}$ such that we have
 \begin{align*} \alpha_i &= \sum_{j_t=0}^{\ord{\alpha_{t,i}}} \cdots \sum_{j_1=0}^{\ord{\alpha_{1,i}}} \lambda_{i,\underline{j}} f^{-a_{1,i}-\cdots- a_{t,i} - j_1 - \cdots- j_t} [\alpha_{1,i}]_f^{j_1} \cdots [\alpha_{t,i}]_f^{j_t} \\
 &= \sum_{j_t=0}^{L c_{t,i}} \cdots \sum_{j_1=0}^{L c_{1,i}} \lambda_{i,\underline{j}} f^{-a_{1,i}-\cdots- a_{t,i} - j_1 - \cdots- j_t} [\alpha_{1,i}]_f^{j_1} \cdots [\alpha_{t,i}]_f^{j_t}
 \end{align*}

 Since $\sum_k L c_{k,i} \leq LCn$, the exponent of $f$ in each term is bounded below by $-(L+1)Cn$. Thus, we can write the sum as $f^{-(L+1)Cn}$ times some element of $D_{R|K}$. 
 
 Since $j_1+\cdots +j_t \leq LCn$ as well, for each of the terms in the sum above we have
 \[ [\alpha_{1,i}]_f^{j_1} \cdots [\alpha_{t,i}]_f^{j_t} \in F^{(c_{1,i} + j_1 m) + \cdots + (c_{t,i} + j_t m)} \subseteq F^{(1+Lm)Cn}.
 \]
For any term in the sum, we have an element like so times a constant, perhaps multiplied on the left by some power of $f$ (to account for the larger common denominator); the exponent of $f$ is at most that in the denominator, so is contained in $F^{LCmn}$. We conclude that 
\[ \alpha_i \in f^{-(L+1)Cn} F^{(1+2Lm)Cn}.\]
Along similar lines, one deduces that
 \[ \beta_i \in  F^{(1+2Lm)Cn}f^{-(L+1)Cn}.\]
 Thus, we have
 \[ 1\in f^{-(L+1)Cn} F^{(1+2Lm)Cn} \delta F^{(1+2Lm)Cn}f^{-(L+1)Cn},\]
 and hence
 \[ f^{2(L+1)Cn} \in F^{(1+2Lm)Cn} \delta F^{(1+2Lm)Cn}.\]
 Hence, $a=2(L+1)C$ and $b=(1+2Lm)C$ serve as the desired constants.
\end{proof}

\begin{lemma}\label{lem:finite cover}
Let $R$ be a domain finitely generated over a field $K$. Suppose that $D_{R|K}$ is a finitely generated $K$-algebra. If there exist $f_1,\dots,f_t\in R$ such that $D_{R_{f_i}|K}$ is linearly simple in the generator filtration for all $i$ and $(f_1,\dots,f_t)=R$, then $D_{R|K}$ is linearly simple in the generator filtration.
\end{lemma}
\begin{proof} Let $\cFb$ denote a generator filtration on $D_{R|K}$. By Proposition~\ref{prop:fastsmallmult}, since each $f_i$ is a nonzerodivisor, there are constants $a_i,b_i$ such that for all $n$ and all nonzero $\delta\in F^n$, we have $f_i^{a_in} \in F^{b_in} \delta F^{b_in}$. Take $m_0$ such that $f_1,\dots,f_t\in F^m$. Let $a$ be the maximum of the $a_i$'s, let $b_0$ be the maximum of the $b_i$'s, and $b=b_0+am_0$. Then
\[ f_i^{an} \in F^{am_0n} f_i^{a_i n} \subseteq F^{am_0n} F^{b_in} \delta F^{b_in} \subseteq F^{amn} F^{b_0n} \delta F^{b_0n} \subseteq F^{bn} \delta F^{bn}.\]

Fix $u_1,\dots,u_t\in R$ such that 
\[ 1 = u_1 f_1 + \cdots + u_t f_t.\]
Pick $m$ such that $u_i \in F^m$ and $f_i\in F^m$ for all $i$. 

For any $n$, we have
\[ 1= (u_1 f_1 + \cdots + u_t f_t)^{tn} = \sum_{i_1+\cdots+i_t=rt} \binom{tn}{i_1,\dots,i_t} u_1^{i_1} \cdots u_t^{i_t} f_1^{i_1}\cdots f_t^{i_t}.\]
By the pigeonhole principle, for each term, at least one of the $i_j$ is greater than $n$, so we can write
\[ 1 = v_{1,n} f_1^n + \cdots + v_{t,n} f_t^n\]
with $v_{i,n}\in F^{2mtn}$ for all $i$. 

Thus, given any nonzero $\delta\in F^n$, we have $f_1^{an},\dots,f_t^{an} \in F^{bn} \delta F^{bn}$, and hence 
\[1\in F^{2mtan} F^{bn} \delta F^{bn} \subseteq F^{Cn} \delta F^{Cn}\]
for $C= 2mta+b$.
\end{proof}

The following lemma is the main technical piece in this section.

\begin{lemma}\label{lem:stdetaleconsts}
Let $R$ be a finitely generated algebra over a field $K$ of characteristic zero, and suppose that $D_{R|K}$ is a finitely generated algebra. Let $S=\left(\frac{R[\theta]}{g(\theta)}\right)_{h}$ be a standard \'etale extension of $R$, where $g$ is a monic polynomial of degree $d$ and $g'(\theta)$ divides $h$. 

Fix a generating set $W$ for $D_{R|K}$ and let $\cFb$ be the associated generator filtration. Let $\cGb$ be the generator filtration on $D_{S|K}$ associated to $W\cup \{ \theta,h^{-1}\}$.

Then there exist constants $A,B$ such that
\[ G^n \subseteq h^{-An} \sum_{i=0}^{d-1} \theta^i F^{Bn}\]
for all $n$.
\end{lemma}
\begin{proof}
    We proceed in three steps.
    
\textbf{Step 1:}  There are constants $A'',B'',C''$ such that for any $\delta\in F^n \cap D_{R|K}^m$ 
and $\ell\in \N$,
 $[\delta,\theta^{\ell}]\in \sum_{\ell<d} \theta^\ell h^{-A''m} F^{n+B''m+C''\ell}$.

Take $e$ such that $c_k\in F^e$ and $f$ such that $w\in F^f$ where $g'(\theta) = w h$. Let $s=de+f$.

Write $W^b_a = \sum_{\ell < d} \theta^{\ell} h^{-a} F^b$.
If $\delta\in W^b_a$, write $\delta = \delta_0 + \theta \delta_1 + \cdots + \theta^{d-1} \delta_{d-1}$ with $\delta_i\in h^{-a} F^b$. Then we have 
\[\theta \delta = \theta \delta_0 + \cdots + \theta^{d-1} \delta_{d-2}  - \sum_{j<d} c_j \theta^j \delta_{d-1} \in W^{b+e}_{a}\] 
\[ c_k \delta = c_k \delta_0 + \theta c_k \delta_1 + \cdots + \theta^{d-1} c_k \delta_{d-1} \in W^{b+e}_{a},\]
and
\[ \frac{-1}{g'(\theta)} \delta = h^{-1} w \delta_0 + \theta h^{-1} w \delta_1 + \cdots + \theta^{d-1}h^{-1} w \delta_{d-1} \in W^{b+f}_{a+1}.\]
That is,
\begin{equation}
\delta\in W^b_a \Longrightarrow \begin{cases} \theta\delta\in W^{b+e}_a \\ c_k \delta\in W^{b+e}_a \\ \frac{-1}{g'(\theta)}\delta\in W^{b+f}_{a+1} \end{cases}\end{equation}

We have \[ [\delta, \theta^{k}] = \sum_{i=1}^{k} \binom{k}{i} \theta^{\ell-i} [\delta]_{\theta}^i,\]
 \[ [\delta, c_k \theta^{k}] = c_k \sum_{i=1}^{\ell} \binom{k}{i} \theta^{\ell-i} [\delta]_{\theta}^i + \sum_{i=0}^{k} \binom{k}{i} \theta^{k-i} [\delta]^{1,i}_{c_k,\theta},\]
 and hence, 
 \[ \sum_{k=0}^d \left( c_k \sum_{i=1}^k \binom{k}{i} \theta^{k-i} [\delta]^i_\theta + \sum_{i=0}^{k} \binom{k}{i} \theta^{k-i} [\delta]^{1,i}_{c_k,\theta}\right) = 0.\]
As the coefficients of $[\delta]_\theta$ sum to $g'(\theta)$, we obtain 
\[ [\delta]_\theta = \frac{-1}{g'(\theta)}  \sum_{k=0}^{d} \left( c_k \sum_{i=2}^k \binom{k}{i} \theta^{k-i} [\delta]_{\theta}^i + \sum_{i=0}^k \binom{k}{i} \theta^{k-i} [\delta]_{c_k, \theta}^{1,i}  \right),\]
we obtain that
\begin{equation}\label{eq:thetaW}\substack{[\delta]_{\theta}^i \in W_a^b \ \text{for} \ 2\leq i \leq d\ \text{and} \\ \ [\delta]_{c_k,\theta}^{1,i} \in W_{a}^{b} \ \text{for} \ 0\leq i \leq d, \ 0 \leq k \leq d-1} \ \Longrightarrow [\delta]_\theta \in W^{b+s}_{a+1}.\end{equation}

Let $\delta\in F^n$. Then for any $t$-tuple $\underline{\alpha}$ of (possibly repeated) elements from $\{c_0,\dots,c_{d-1}\}$ we have
\begin{equation}\label{eq:p0} [\delta]_{\underline{\alpha}}^{\underline{1}} \in F^{n+te},\end{equation}
where $\underline{1}$ is a vector of ones of length equal to that of $\underline{\alpha}$.

Let $m=\mathrm{ord}(\delta)-1\leq Cn-1$. For a $q=(m-p)$-tuple $\underline{\alpha}$ of (possibly repeated) elements from $\{c_0,\dots,c_{d-1}\}$ applying \eqref{eq:thetaW} to $[\delta]^{p-1,\underline{1}}_{\theta, \underline{\alpha}}$, we have 
$\mathrm{ord}([\delta]^{p,\underline{1}}_{\theta, \underline{\alpha}})=0$ and hence
\begin{equation}\label{eq:pind} [\delta]^{p-1,1,\underline{1}}_{\theta, c_k, \underline{\alpha}}\in W^b_a \ \text{for all} \ 0\leq k \leq d-1  \Longrightarrow [\delta]^{p,\underline{1}}_{\theta, \underline{\alpha}}\in W^{b+q}_{a+1}.\end{equation}

Thus, by induction on $p$, with base case \eqref{eq:p0} and inductive step \eqref{eq:pind}, for $\delta\in F^n$ of order $m$ and any $q=(m-p)$-tuple $\underline{\alpha}$ of (possibly repeated) elements from $\{c_0,\dots,c_{d-1}\}$ we have 
\begin{equation}\label{eq:p} [\delta]^{p,\underline{1}}_{\theta,\underline{\alpha}} \in W^{n+me+pq}_{p}.\end{equation}

Now we claim that for $\delta\in F^n$ of order $m$ and any $q$-tuple $\underline{\alpha}$ (with $q\leq m-p$) of  (possibly repeated) elements from $\{c_0,\dots,c_{d-1}\}$ we have 
 \begin{equation}\label{eq:ind-delta}
 [\delta]^{p,\underline{1}}_{\theta,\underline{\alpha}} \in W^{n+me+(2m-2q-p)s}_{2m-2q-p}.
 \end{equation}
To see this we proceed by descending induction on $t=p+q$ starting with the base case $t=m$, which is just \eqref{eq:p}. For the inductive step, we proceed now by a nested induction on $p$, with the base case $p=0$ covered by \eqref{eq:p0}. For the inductive step of the induction on $p$, given $p\geq 1$ and a $q$-tuple $\alpha$ of (possibly repeated) elements from $\{c_0,\dots,c_{d-1}\}$ we apply \eqref{eq:thetaW} to $\delta'=[\delta]_{\theta,\underline{\alpha}}^{p-1,\underline{1}}$, observing that each $[\delta']_\theta^i=[\delta]_{\theta,\underline{\alpha}}^{p-1+i,\underline{1}}$ for $2\leq i\leq d$ and $[\delta']_{c_k,\theta}^{1,i}=[\delta]_{\theta,c_k,\underline{\alpha}}^{p-1+i,1,\underline{1}}$ for $0\leq i \leq d$, $0 \leq k \leq d-1$ is subject to the induction hypothesis, and in particular \[[\delta']_\theta^i\in W^{n+me+(2m-2q-p+1-i)s}_{2m-2q-p+1-i}\subseteq W^{n+me+(2m-2q-p-1)s}_{2m-2q-p-1}\] and \[[\delta']_{c_k,\theta}^{1,i}\in W^{n+me+(2m-2q-2-p+1-i)s}_{2m-2q-2-p+1-i}\subseteq W^{n+me+(2m-2q-p-1)s}_{2m-2q-p-1}.\]
Then, by \eqref{eq:thetaW}, we get that $\delta = [\delta']_\theta \in W^{n+me+(2m-2q-p)s}_{2m-2q-p}$, completing the inductions.

Applying \eqref{eq:ind-delta} with $q=0$, we get that $[\delta]_{\theta}^p \in W^{n+m(e+2s)}_{2m}$ for all $p$. Then 
\[[\delta,\theta^\ell]=\sum_{j=1}^\ell \binom{\ell}{j} \theta^{\ell-j} [\delta]^j_\theta\in \theta^\ell W^{n+(e+2s)m}_{2m}\subseteq W^{n+(e+2s)m+ e \ell}_{2m},\]
so the constants $A''=2$, $B''=e+2s$, and $C''=e$ suffice.

\textbf{Step 2:} There are constants $A',B',C',E'$ such that for any $\delta\in F^n \cap D^m$ and $v,\ell\in \N$ we have ${\sum_{\ell<d} [\delta, h^{-v}\theta^\ell}]  \in \sum_{\ell<d} \theta^\ell h^{-v-A'm} F^{n+B'm+C'\ell + E'}$.

Recall that for all $j\in\Z$,
we have \[[\delta,h^{v}] = \sum \binom{v}{j} h^{v-j} [\delta]_h^j\]
where the sum runs from $j=1$ to $\min\{e,v\}$ if $v> 0$ and to $\mathrm{ord}(\delta)$ if $v<0$.
Thus, by a straightforward induction, for $v>0$
\[ [\delta]_{h}^v = \sum_{j=0}^{v-1} (-1)^j \binom{v}{j} h^j [\delta,h^{v-j}].\]
In particular,
\[ [\delta,h^{-v}] = \sum_{i=1}^{e} n_i h^{-v-i} [\delta,h^i]\]
for some constants $n_i\in \Z$.
Write $h=c'_{d-1}\theta^{d-1} + \cdots + c'_1 \theta + c'_0$ with $c'_i\in R$ and let $c'_k\in F^{e'}$. 
Writing $h^i=\sum_{\ell<d} c'_{d-1,i}\theta^{d-1} + \cdots + c'_{1,i} \theta + c'_{0,i}$, we have by induction that $c'_{k,i}\in F^{ie'+(d-1)(i-1)e} \subseteq F^{s' i}$ with $s'=e' + (d-1)e$.

Let $\delta\in F^n\cap D^m_{R|K}$. Then
 \[\begin{aligned}
 \relax[\delta,h^i] &= \sum_{\ell < d} [\delta, \theta^\ell c'_{\ell,i}] =  \sum_{\ell < d} \theta^\ell [\delta,c'_{\ell,i}] + \sum_{\ell < d}  [\delta, \theta^\ell] c'_{\ell,i}\\ &\subseteq \sum_{\theta<\ell} \theta^\ell F^{n+s'i} + W^{n+B'' m +C'' (d-1)}_{A'' m} F^{s' i} \subseteq W^{n+B'' m+ C'' (d-1) + s' i}_{A''m}.
 \end{aligned} \]
 
 Then, we have
\begin{equation} [\delta,h^{-v}] \subseteq \sum_{i=1}^{m} h^{-v-i} W^{n+B'' m+ C'' (d-1) + s' i}_{A'' m} \subseteq W^{n+(B''+s)m + C'' (d-1)}_{v+ (A''+1)m}.\end{equation}

Consequently, for $\mu = \sum_{\ell<d} \theta^\ell h^{-a} \delta_i \in W_a^b$ with $\delta_i\in F^b \cap D^m_{R|K}$, we have
\begin{equation} \begin{aligned} \relax[\mu,h^{-v}] &= \sum_{\ell<d} [\theta^\ell h^{-a} \delta_i, h^{-v}] =  \sum_{\ell<d} \theta^\ell h^{-a} [\delta_i, h^{-v}] \\& \in \sum_{\ell<d} \theta^\ell h^{-a} W^{b+(B''+s)m+ C'' (d-1)}_{v+ (A''+1)m} \subseteq W^{b+(B''+s)m + (d-1)(C''+e)}_{a+v+(A''+1)m}.
\end{aligned}\end{equation}

Then, applying the formula above we get
$$\begin{aligned}\relax
[\delta, &h^{-v} \theta^\ell]  =   h^{-v} [\delta,\theta^{\ell}] + \theta^{\ell} [\delta,h^{-v}]  + [ [\delta, \theta^{\ell}] , h^{-v} ] \\
&\in h^{-v} W^{n+B''m+C''\ell}_{A''m} + \theta^\ell W^{n+(B''+s)m+ C'' (d-1)}_{v+ (A''+1)m} + W^{n+B''m+(B''+s)m + C''\ell+ (C''+e)(d-1)}_{A''m+v+(A''+1)m}\\&\subseteq W^{n+(2B'' +s)m + C''\ell+ (C''+e)(d-1)}_{v+ (2A''+1)m},
 \end{aligned}$$
so we can take $A'=2A''+1$, $B'=2B'' +s+de$, $C'=C''$, and $E'=(C''+e)(d-1)$.

\textbf{Step 3:} There are constants $A,B$ such that $G^n \subseteq \sum_{\ell<d} \theta^\ell h^{-An} F^{Bn}$.

We claim that $A=A'\gamma$ and $B=B'\gamma +1 + C'(d-1)+E'$ work. Note that $A\geq 1$ and $B\geq E'\geq e$. We proceed by induction on $n$, with $n\leq 1$ clear. We can write any element of $G^n$ as a sum of elements of the form 
\[ \begin{cases} \theta \mu &  \mu\in G^{n-1}, \\ h^{-1} \mu & \mu\in G^{n-1}, \text{and} \\ \delta \mu & \delta\in F^{a}, \ \mu\in G^{n-a}, \ a>0. \end{cases}\]

For $ \theta \mu$ with  $\mu\in G^{n-1}$, we have $\mu \in W^{B(n-1)}_{A(n-1)}$ by the inductive hypothesis. Then ${\theta \mu \in W^{B(n-1) + e}_{A(n-1)} \subseteq W^{Bn}_{An}}$, as required.

For $h^{-1}\mu$ with $\mu\in G^{n-1}$, we again have $\mu \in W^{B(n-1)}_{A(n-1)}$, and $h^{-1}\mu \in W^{B(n-1)}_{A(n-1)+1} \subseteq W^{Bn}_{An}$, as required.

For $\delta \mu$ with $\delta\in F^{a}$ and $\mu\in G^{n-a}$, we have $\mu\in W^{B(n-a)}_{A(n-a)}$, so we can write $\mu$ as a sum of elements of the form $\theta^\ell h^{-v} \alpha$ with $\ell<d$, $v<B(n-a)$, and $\alpha\in F^{A(n-a)}$. Then $\delta \mu$ can be written as a sum of elements of the form $\delta \theta^\ell h^{-v} \alpha$. We have
\[ \delta \theta^\ell h^{-v} \alpha = \theta^\ell h^{-v} \delta  \alpha + [\delta, \theta^\ell h^{-v}] \alpha.\]
We have $ \theta^\ell h^{-v} \delta  \alpha \in W^{0}_{v} F^a F^{B(n-a)} \subseteq W^{a+B(n-a)}_v \subseteq W^{Bn}_{An}$.
Also, since $\delta\in \gamma a$, we have 
\[ [\delta, \theta^\ell h^{-v}] \in W^{a+ B' \gamma a + C' \ell + E'}_{v+A' \gamma a} \subseteq W^{a + B' \gamma a + C'(d-1) + E'}_{A(n-a)+A' \gamma a},\]
so
\[ [\delta, \theta^\ell h^{-v}] \alpha \in W^{a + B' \gamma a + C'(d-1) + E' + B(n-a)}_{A(n-a)+A' \gamma a}.\]
By choice of $A$ and $B$, we have $a + B' \gamma a + C'(d-1) + E' + B(n-a) \leq Bn$ and ${A(n-a)+A' \gamma a \leq An}$. This shows that $\delta\mu\in W^{Bn}_{An}$, completing the induction, and completing the proof.
\end{proof}

\begin{proposition}\label{lemma:etale}
Let $R$ be a domain finitely generated over a field $K$. Suppose that $D_{R|K}$ is a finitely generated $K$-algebra, and that $D_{R|K}$ is linearly simple in the generator filtration. Let $S$ be an \'etale extension of $R$ that is also a domain. Then $D_{S|K}$ is linearly simple in the generator filtration.
\end{proposition}
\begin{proof}
Note first that by Lemma~\ref{Localizing Things}(2), there is an isomorphism
\[ S\otimes_R D_{R|K} \xrightarrow{\sim} D_{S|K}.\]
 In particular, every element of $D_{S|K}$ can be written as a sum of products of elements of $S$ and (extensions of elements of) $D_{R|K}$, so $D_{S|K}$ is a finitely generated $K$-algebra, and hence it makes sense to discuss the generator filtration.

By the structure theory for \'etale morphisms (see, e.g., \cite[{Tag 00UE}]{stacks-project}), for every ${\mathfrak{q}\in \Spec(S)}$, there is some $g$ such that $R\to S_g$ is standard \'etale. In particular, 
\[\{{g\in S} \ | \ {R\to S_g} \, \text{ is standard \'etale}\}\] generates the improper ideal, so there exist $g_1,\dots,g_t$ such that $R\to S_{g_i}$ is standard \'etale and $(g_1,\dots,g_t)=S$. By Lemma~\ref{lem:finite cover}, we then reduce to the case of a standard \'etale extension.

Let $T=\frac{R[\theta]}{(g)}$ and $S=T_h$ for some polynomial $g$ that is monic in $\theta$ of degree $d$ and $h$ a multiple of $g'(\theta)$. 
Fix a generating set $W$ for $D_{R|K}$ and let $\cFb$ be the associated generator filtration. Fix a constant of linear simplicity for $\cFb$, i.e., a constant $C$ such that for $\delta\in F^n\smallsetminus 0$, we have $1\in F^{Cn} \delta F^{Cn}$.
Let $\cGb$ be the generator filtration on $D_{S|K}$ associated to $W\cup \{ \theta,h^{-1}\}$. Fix $e$ such that $h\in G^e$. By Lemma~\ref{lem:stdetaleconsts}, there are constants $A,B$ such that 
\[ G^n \subseteq h^{-An} \sum_{i=0}^{d-1} \theta^i F^{Bn}\]
for all $n$. Let $c$ be such that the coefficients of $g(\theta)$ are in $F^c$.

Let $s\in G^n \cap D^0_{S|K}$. Write
\[ h^{An} s = \delta_0 + \theta \delta_1 + \cdots + \theta^{d-1} \delta_{d-1}\]
in $T\otimes_K D_{R|K}$ for $\delta_0,\delta_1,\dots,\delta_{d-1} \in F^{Bn}$.
On the other hand, there are elements ${r_0,r_1,\dots,r_{d-1}\in D^0_{R|K}}$ such that 
\[ t= h^{An} s = r_0 + \theta r_1 + \cdots + \theta^{d-1} r_{d-1},\]
so, since $T$ is a free $R$-module, we have $r_i=\delta_i\in F^{Bn}\cap D^0_{R|K}$ for all $i$.
Then one can write 
\[ t \theta^i = \sum_{i,j} r_{i,j} \theta^j \]
with $r_{i,j}\in F^{Bn+c} \cap D^0_{R|K}$. Setting $M=[r_{i,j}]$, we have that $t$ is a root of the integral equation $\det(X I_{d\times d} - M) =0$, which has coefficients in $F^{Bdn+Cd} \cap D^0_{R|K}$. Factoring out powers of $X$ if necessary, moving the constant term to one side, evaluating at $t$, there exists $t'\in F^{Bdn+cd}$ such that 
\[tt'\in (F^{Bdn+cd} \cap D^0_{R|K})\smallsetminus 0,\]
and hence $1\in F^{C({Bdn+cd})} tt' F^{C({Bdn+cd})}$. Thus,
\begin{align*}
1&\in F^{C({Bdn+cd})} tt' F^{C({Bdn+cd})} \subseteq F^{C({Bdn+cd})} t F^{Bdn+cd} F^{C({Bdn+cd})} \\&\subseteq 
F^{C({Bdn+cd})} h^{An} s F^{Bdn+cd} F^{C({Bdn+cd})}
\subseteq 
G^{C({Bdn+cd})} G^{Aen} s G^{Bdn+cd} G^{C({Bdn+cd})}\\&
\subseteq  
G^{(Bd+Ae)n + Ccd} s G^{(C+1)Bd n + (C+1) cd}.
\end{align*}
Then, we can take some $L$ such that $Ln \geq \max\{ (Bd+Ae)n + Ccd, (C+1)Bd n + (C+1) cd\}$ for all $n\in \N$ and we have
\[ 1 \in G^{Ln} s G^{Ln}.\]
By Lemma~\ref{reduce to R bar}, we conclude that $(D_{S|K}, \cGb)$ is linearly simple. This concludes the proof.
\end{proof}

\begin{proposition}\label{prop:smooth}
 If $R$ is a domain that is a smooth algebra over a field $K$ of characteristic zero, then $D_{R|K}$ is linearly simple in the generator filtration.
\end{proposition}
\begin{proof}
Note first that by \cite[1.4(a)]{SmithStafford} (see also \cite[Chapter 3, Theorem 2.3]{Bjork}, $D_{R|K}$ is finitely generated, and hence it makes sense to discuss the generator filtration.

By the structure theory for smooth morphisms (see, e.g., \cite[\href{https://stacks.math.columbia.edu/tag/054L}{Tag~054L}]{stacks-project}) for every ${\mathfrak{q}\in \Spec(R)}$, there is some $f\notin \mathfrak{q}$ such that $K \to R_f$ factors as an \'etale extension of a polynomial ring over $K$. As in the proof of Lemma~\ref{lemma:etale}, there exist $f_1,\dots f_t$ such that $K \to R_{f_i}$ factors as an \'etale extension of a polynomial ring over $K$ and ${(f_1,\dots,f_t)=R}$.
By Lemma~\ref{lem:finite cover}, we reduce to the case of an \'etale extension of a polynomial ring and by Lemma~\ref{lemma:etale}, we reduce to the case of a polynomial ring over a field of characteristic zero $R=K[x_1,\dots,x_n]$. In this case, the generator filtration on $D_R$ with generating set $\{x_1,\dots,x_n,\partial_1,\dots,\partial_n\}$ is identical to the Bernstein filtration with slope $2$ with the standard grading on $R$; this is then linearly simple by  \cite[Proposition~4.22]{AHJNTW2}.
\end{proof}

\begin{theorem}\label{SBSimpliesLS}
Let $R$ be a finitely generated $\NN$-graded algebra over a field $K$ of characteristic zero. Let $f$ be a nonzerodivisor in the singular locus of $R$. If $f$ admits a sandwich Bernstein-Sato polynomial with no nonnegative integer roots, then $D_{R|K}$ is linearly simple in any Bernstein filtration.
\end{theorem}
\begin{proof}
Let $R$ be a finitely generated $\N$-graded algebra over $K$ and $f$ be an element of the singular locus. Since $R_f$ is a finitely generated $K$-algebra that is regular, $R_f$ is a smooth $K$-algebra. Then by Proposition~\ref{prop:smooth}, $D_{R_f|K}$ is linearly simple in the generator filtration. Fix a Bernstein filtration $\cBb$ on $R$. By Proposition~\ref{prop:simple-filtrations}, the localized Bernstein filtration on $D_{R_f|K}$ is linearly simple. Then by Proposition~\ref{prop:fastsmallmult}, there are constants $a,b$ such that for all $n$ and all nonzero $\delta\in B^n$, one has $f^{an}\in B^{bn} \delta B^{bn}$. Fix a sandwich Bernstein-Sato functional equation
\[ \sum_i \alpha_i(s) f^{s+1} \beta_i(s) = b(s) f^s\]
for which $b(t)\neq 0$ for all $t\in \N$.
Let $m\in \N$ be such that every component of each $\alpha_i(s)$ and $\beta_i(s)$ lies in $B^{m}$. Then applying the functional equation repeatedly one obtains that
\[ 1\in B^{mt} f^t B^{mt}\]
for all $t\in \N$ and hence
\[ 1\in B^{amn} f^{an} B^{amn}.\]
Thus, we have
\[ 1\in B^{amn} f^{an} B^{amn} \subseteq B^{amn} B^{bn} \delta B^{bn} B^{amn} \subseteq B^{(am+b)n} \delta B^{(am+b)n}.\]
It follows that $D_{R|K}$ is linearly simple in the Bernstein filtration.
\end{proof}

We are also able to apply similar methods to give an effective condition for positive differential signature.

\begin{theorem}\label{thm:diffsig} Let $R$ be a finitely generated $\N$-graded algebra over a field $K$ of characteristic zero. Suppose that there exists some nonzerodivisor $f$ in the singular locus of $R$ that admits a Bernstein-Sato polynomial with no nonnegative integer roots. Then $\mathrm{s}_K^\mathrm{diff}(R)>0$.
\end{theorem}
\begin{proof}
As in the proof of Theorem~\ref{SBSimpliesLS}, by the hypotheses, we have that for some Bernstein filtration $\cBb$ of $D_{R|K}$, there are constants $a,b$ such that for all nonzero $\delta\in B^n$, we have $f^{an} \in B^{bn} \delta B^{bn}$. In particular, if $r\in R$ is a nonzero element of degree at most $n$, then $r\in B^n$, and we can write
\[ f^{an} = \sum_i \alpha_i r \beta_i\]
in $D_{R|K}$ with $\alpha_i,\beta_i\in D^{Ln}_{R|K}$ for some constant $L$ independent of $n$. Applying both sides as operators to $1\in R$, we get
\[ f^{an} = \sum_i \alpha_i r \beta_i(1) = \sum_i \alpha_i (r \beta_i(1)) = \sum_i (\alpha_i \overline{\beta_i(1)})(r),\]
where $\overline{\beta_i(1)}$ denotes the operator of order zero corresponding to multiplication by $\beta_i(1)$.
Note that $\sum_i \alpha_i \overline{\beta_i(1)} \in D^{Ln}$.

Now, fix a Bernstein-Sato function equation for $f$ with no nonnegative integer roots and let $m$ be such that each coefficient of the operator in the expression has order at most $m$. Applying this functional equation repeatedly, one finds that 
\[1\in D^{mt} \cdot f^t\]
for all $t\in\N$. Thus, given a nonzero element $r\in R$ of degree at most $n$, we have
\[ 1\in D^{amn} \cdot f^{an} \subseteq D^{amn} D^{Ln} \cdot r = D^{(am+L)n} \cdot r.\]
It follows that 
\[ \mathfrak{m}^{\langle (am+L)n+1\rangle_K} \subseteq [R]_{>n}\quad \text{and hence} \quad \ell(R/\mathfrak{m}^{\langle (am+L)n+1\rangle_K}) \geq \ell([R]_{\leq n})\]
for all $n\in \N$. Then the conclusion follows from a basic comparison with the Hilbert function.
\end{proof}

\section{Comparison of Bernstein-Sato polynomials and sandwich Bernstein-Sato polynomials}

In this section, we compare the Bernstein-Sato polynomials of elements with sandwich-Bernstein-Sato polynomials. The first thing we note is that the sandwich-Bernstein-Sato polynomial, when it exists, is always a multiple of the standard Bernstein-Sato polynomial. We recall that we denote the sandwich Bernstein-Sato polynomial of $f$ by $\SBS{f}{R}(s)$; we denote the standard Bernstein-Sato polynomial of $f$ by $\BS{f}{R}(s)$.

\begin{proposition}\label{prop-divides}
Let $R$ be a $K$-algebra where $K$ is a field of characteristic zero, and let $f \in R$. Then $\BS{f}{R}(s)$ divides $\SBS{f}{R}(s)$.
\end{proposition}
\begin{proof}
Let $\SBS{f}{R}(s) = b(s)$, and let $\sum_i \alpha_i(s) f^{s+1} \beta_i(s) = b(s)f^s$ be a Sandwich-Bernstein-Sato functional equation corresponding to $b(s)$. This equation is an equality of operators in $D_{R|K}$, so we can evaluate both sides at $1\in R$ and rearrange\footnote{To add clarity, if $r$ is a ring element of $R$ we denote the ``multiply by $r$'' operator in $D_{R|K}$ as $\overline{r}$.}:
\begin{align*}
    \left(\sum_i \alpha_i(s) \overline{f^{s+1}} \beta_i(s)\right)(1)            &=  \overline{b(s)f^s}(1)\\
    \sum_i \alpha_i(s) \left(f^{s+1} \cdot\beta_i(s)(1)\right)                  &=  (b(s)f^s)\cdot 1\\
    \sum_i \alpha_i(s)\overline{\beta_i(s)} \left(f^{s+1}\right)                &=  \overline{b(s)}(f^s)\\
    \left(\sum_i \alpha_i(s)\overline{\beta_i(s)}\right) \left(f^{s+1}\right)   &=  \overline{b(s)}(f^s)
\end{align*}
This yields a Bernstein-Sato functional equation: with $\delta(s) = \sum_i \alpha_i(s)\overline{\beta_i(s)}$ and $b(s)$ as above, the following equality is true when considered as an equation in $R$ by evaluating the left-hand-side at $f^{s+1}$: $\delta(s)f^{s+1} = b(s) f^s$. This means that $b(s)$ is an element of the ideal  generated by $\BS{f}{R}(s)$ in  $K[s]$. Therefore $\BS{f}{R}(s)$ divides $\SBS{f}{R}(s)$.
\end{proof}

We can now state a simple criterion for Bernstein's inequality in terms of sandwich Bernstein-Sato equations.

\begin{corollary}\label{thm-main-BI}
Let $R$ be a finitely generated $\N$-graded algebra over a field $K$ of characteristic zero. Suppose that there exists some nonzerodivisor $f$ in the singular locus of $R$ that admits a sandwich Bernstein-Sato polynomial with no nonnegative integer roots. Then $R$ is a Bernstein algebra; in particular, if $M$ is nonzero $D_R$-module equipped with a filtration $\cFb$ compatible with some Bernstein filtration on $D_R$, then $\dim(\cFb)\geq \dim(R)$.
\end{corollary}
\begin{proof} By Theorem~\ref{SBSimpliesLS}, any Bernstein filtration on $D_R$ is linearly simple. By Proposition~\ref{prop-divides}, $f$ also satisfies a Bernstein-Sato polynomial with no nonnegative integer roots. By Theorem~\ref{thm:diffsig}, $R$ has positive differential signature. Then, by Proposition~\ref{prop:posdiff}, $R$ is a Bernstein algebra.
\end{proof}

\begin{proposition}\label{prop-equal}
Let $K$ be a field of characteristic zero and $R$ be a $K$-algebra. Suppose that $D_{R|K}$ is generated as a ring by $D^1_{R|K}$.  Then $\BS{f}{R}(s)=\SBS{f}{R}(s)$ for all $f\in R$.

In particular, for $R=K[x_1,\dots,x_n]$, Bernstein-Sato polynomials and sandwich-Bernstein-Sato polynomials coincide.
\end{proposition}
\begin{proof}
We already have that $\BS{f}{R}(s)$ divides $\SBS{f}{R}(s)$. Let $b(s)=\BS{f}{R}(s)$, and take a Bernstein-Sato functional equation
\[ P(s) \cdot f^{s+1} = b(s) f^s, \quad \text{so} \quad P(t) \cdot f^{t+1} = b(t) f^t \ \text{in} \ R_f \ \text{for all} \  t\in \Z. \]
It suffices to exhibit a sandwich functional equation with the same polynomial $b(s)$. For clarity, we will write $\overline{r}$ for the element of $D^0_{R|K}$ given by multiplication by $r$.
Since $D_{R|K}$ is generated by $D^1_{R|K}$, any element of $D_{R|K}$ can be written as a sum of products of multiplication operators and derivations. Using the relations $\theta \overline{r} =  \overline{r} \theta + \overline{\theta(r)}$ for $\overline{r}\in D^0_{R|K}$ and $\theta\in \Der_{R|K}$, for any $t\in \Z$, we can write
\[ P(t) = \sum_j r_j \theta_{j,1} \cdots \theta_{j,d_j} t^{e_j}\]
for some $\overline{r_j}\in D^0_{R|K}$ and $\theta_{j,i}\in \Der_{R|K}$. But, again using that $[\theta,\overline{r}] = \overline{\theta(r)}$ for $\overline{r}\in D^0_{R|K}$ and $\theta\in \Der_{R|K}$, we have
\[ \sum_j r_j t^{e_j} \left[ \theta_{j,1} \left[ \cdots \left[ \theta_{j,d_j}, \overline{f}^{t+1} \right] \cdots \right] \right] = \overline{(P(t) \cdot f^{t+1})} \]
in $D_{R|K}$. We then have
\[ \sum_j r_j t^{e_j} \left[ \theta_{j,1} \left[ \cdots \left[ \theta_{j,d_j}, \overline{f}^{t+1} \right] \cdots \right] \right] = \overline{b(t) f^t} = b(t) \overline{f}^t\]
in $R_f$ for all $t\in \Z$, so
\[ \sum_j r_j s^{e_j} \left[ \theta_{j,1} \left[ \cdots \left[ \theta_{j,d_j}, \overline{f}^{s+1} \right] \cdots \right] \right] = b(s) \overline{f}^s\]
 is a sandwich Bernstein-Sato functional equation with polynomial $b(s)$.
\end{proof}

\begin{remark}
    It follows from the previous proposition that if some Bernstein-Sato polynomial and sandwich Bersntein-Sato polynomial disagree for an element, then the ring of differential operators is not generated by derivations and multiplication maps.
\end{remark}

However, in singular rings, the two polynomial can differ.

\begin{example}\label{ex-f=y}
Consider the element $f=y$ in $R=\mathbb{C}[x^2,x^3,y,xy]$. Every differential operator on $S=\mathbb{C}[x,y]$ that preserves the subring $R$ is a differential operator on $R$. Indeed, $\Frac(R)=\Frac(S)$, every operator on $S$ extends uniquely to an operator on $\Frac(S)$ by \ref{Localizing Things}(1), and every operator on $\Frac(R)$ that preserves $R$ restricts to an operator on $R$ by \cite[Corollary~2.2.6]{Masson}. Writing 
\[ R = \bigoplus_{\substack{i,j\geq 0\\ (i,j) \neq (1,0)}} K \cdot x^i y^j,\]
we have 
\[ (1- x \partial_x) \partial_y (R) \subseteq (1- x \partial_x) \partial_y 
 \Big(\bigoplus_{\substack{i,j\geq 0\\ (i,j) \neq (1,0)}}  K \cdot x^i y^j \Big) \subseteq \bigoplus_{\substack{i,j\geq 0\\ i \neq 1}} K \cdot x^i y^j \subseteq R, \]
so $(1- x \partial_x) \partial_y$ is a valid differential operator on $R$.

We observe that there is a Bernstein-Sato functional equation
\[ (1- x \partial_x) \partial_y \cdot y^{s+1} = (s+1) y^s,\]
so $\BS{f}{R}(s) = s+1$.
However, $s$ divides $\SBS{y}{R}(s)$. Indeed, there is a two-sided ideal $\cJ$ of $D_{R|\mathbb{C}}$ that contains $y$ but not $1$, since $H^1_{(x^2,x^3,y,xy)}(R)\cong S/R\cong \mathbb{C}$, which has annihilator $(x^2,x^3,y,xy)$ as an $R$-module; the annihilator of $H^1_{(x^2,x^3,y,xy)}(R)$ as a $D_{R|\mathbb{C}}$ is a two-sided ideal. 
Thus, given an SBS-functional equation
\[ \sum_i \alpha_i(s) y^{s+1} \beta_i(s) = b(s) y^s,\]
evaluating at $s=0$, the left-hand side is in $\cJ$, and since $y^0=1$, we must have $b(0)=0$.

Thus $\BS{f}{R}(s)\neq \SBS{f}{R}(s)$ in this example.
\end{example}

To generate more examples in which the two notions differ, we now establish a method for detecting roots of sandwich-Bernstein-Sato polynomials via reduction modulo $p$. The following is a two-sided analogue of the result of Musta\c{t}\u{a}, Takagi, and Watanabe \cite{MTW}.

\begin{lemma} 
Let $T$ be a finitely generated $\Z$-algebra and $R=\Q \otimes_{\ZZ} T$. Let $f \in T$, and $b(s)\in \Q[s]$ be a polynomial of a sandwich Bernstein-Sato functional equation for the image of $f$ in $R$. 
Then there exists $N\in \N$ such that for any prime $p>N$, and any two sided ideal $\cJ \subset \End_{(T/pT)^p}(T/pT)$, we have that
\[ f_p^{v} \notin \cJ\quad \text{and}\quad  f_p^{v+1} \in \cJ \quad \text{implies} \quad b(v) \equiv 0 \text{ modulo } p,\]
where $f_p$ denotes map of multiplication by the image of $f$ modulo $p$. 
\end{lemma}
\begin{proof}
Let $R$, $T$, and $f$ be as above. Fix a two-sided functional equation for $f$ over $D_{R|\Q}$. Since differential operators for finitely generated algebras localize, we have $D_{R|\Q}\cong \Q\otimes_{\Z} D_{T|\Z}$, so the functional equation can be written as
\[ \sum_i \frac{1}{m_1}\alpha_i(s) f^{s+1} \beta_i(s) = \frac{1}{m_2} a(s) f^s\]
with $\alpha_i,\beta_i\in D^n_{T|\Z}[s]$, $a(s) = m_2 b(s)\in \Z[s]$, and $m_1,m_2\in \Z$. We claim that the number ${N=\max\{m_1,m_2,n\}}$ satisfies the statement.

Given $v\in \N$, since $m_1,m_2$ are units in $\F_p$, by Proposition~\ref{prop:charpdiff}(1) we can evaluate the equation at $s=v$ and reduce modulo $p$ to obtain an equation in $D_{T/pT|\F_p}$:
\begin{equation}\label{eq:2sided} \sum_i \frac{1}{m_1}\alpha_i(v) f_p^{v+1} \beta_i(v) = \frac{1}{m_2} a(v) f_p^v,\end{equation}
where $f_p$ denotes the image of $f$ in $T/pT= D^0_{T/pT}$.
Note that, by Proposition~\ref{prop:charpdiff}(2) \[ \alpha_i(v),\beta_i(v)\in D^{n}_{T/pT|\F_p} \subseteq D^{p-1}_{T/pT|\F_p} \subseteq \End_{(T/pT)^p}(T/pT).\]
Thus, if $f_p^{v+1}\in \cJ$, the left-hand side of \eqref{eq:2sided} is in $\cJ$; since $f_p^{v}\notin \cJ$, we must have $a(v)=0$ in $\F_p$, which implies $b(v)=0$ in $\F_p$.
\end{proof}

\begin{proposition}\label{prop:charp-to-roots}
Let $T$ be a finitely generated $\Z$-algebra and $R=\Q \otimes_{\ZZ} T$. Let $f \in T$, and $b(s)\in \Q[s]$ be a polynomial of a sandwich Bernstein-Sato functional equation for the image of $f$ in $R$.

Suppose that there exists a polynomial $r(s)\in \Q[s]$ and an infinite set of prime integers $P$ such that for every $p\in P$, there is a two-sided ideal $\cJ_p$ in $\End_{(T/pT)^p}(T/pT)$ such that \[r(p)\in \ZZ_{\geq 0}, \ f_p^{r(p)} \notin \cJ_p, \ \text{and} \ f_p^{r(p)+1} \in \cJ_p,\]
where the elements $f_p^{r(p)}, f_p^{r(p)+1}\in \End_{(T/pT)^p}(T/pT)$ denote the multiplication maps by the respective ring elements.
Then $r(0)$ is a root of $b(s)$.
\end{proposition}
\begin{proof}
Fix a number $N$ as in the conclusion of the previous proposition; without loss of generality, suppose that every denominator of each coefficient of $r(s)$ is less than $N$.

For any prime $p\in P$ larger than $N$, by the proposition, we have $b(r(p)) \equiv 0 \pmod{p}$. Since $b(r(p))\equiv b(r(0)) \pmod{p}$, we obtain that $b(r(0))\equiv 0 \pmod{p}$ for infinitely many prime numbers $p$. This implies that that the numerator of $b(r(0))$ is divisible by infinitely many primes, which implies that $b(r(0))=0$.
\end{proof}

\begin{example}
Let $R=K[x^2,xy,y^2]$ and consider $f=xy$ in $R$.
 We claim that for $K$ of characteristic zero, $\BS{f}{R}(s)=(s+1)^2$ and $\SBS{f}{R}(s) = (s+1)^2 (s+2)$.

    To see that $\BS{f}{R}(s)=(s+1)^2$, we observe that $R$ is a differentially extensible direct summand of $S=K[x,y]$ in the sense of \cite[Definition~6.1]{BJNB}; this follows from \cite[Proposition~6.4]{BJNB}. Then since $\BS{f}{S}(s)=(s+1)^2$, by \cite[Theorem~6.11]{BJNB}, we also have $\BS{f}{R}(s)=(s+1)^2$.

    To see that $\SBS{f}{R}(s) = (s+1)^2 (s+2)$, first we show that $\SBS{f}{R}(s) \ | \ (s+1)^2 (s+2)$ by exhibiting a functional equation with this corresponding polynomial. Observe that as in Example~\ref{ex Veronese}, the ring of differential operators on $R$ can be identified with sums of differential operators on $S$ of even degree.
    Then, we have the sandwich Bernstein-Sato functional equation
    \[\begin{aligned} (s+2) \partial_x \partial_y f^{s+1} + (s+2) f^{s+1} \partial_x \partial_y - \partial_x \partial_y f^{s+1} (x \partial_x + y \partial_y) + x \partial_y f^{s+1} \partial_x^2 &+ y \partial_x f^{s+1} \partial_y^2 \\&= (s+1)^2(s+2) f^s. \end{aligned}\]

On the other hand, by Proposition~\ref{prop-divides}, we have $(s+1)^2 \ |\ \SBS{f}{R}(s)$, so it only remains to show that $-2$ is a root of $\SBS{f}{R}(s)$; for this, we will apply
Proposition~\ref{prop:charp-to-roots}. 

To this end, let $K$ be a perfect field of characteristic $p>0$. We will first exhibit an $R^p$-module decomposition of $R$. Let $I=R x + Ry \subseteq S$;  $I$ is isomorphic to the ideal $(x^{2},xy)R$.
Note first that
\[ R = \bigoplus_{\substack{0\leq a,b\\ a+b \ \text{even}}} K x^a y^b \quad \text{and} \quad I = \bigoplus_{\substack{0\leq a,b\\ a+b \ \text{odd}}} K x^a y^b. \]  

There is a direct sum decomposition of $R$ as an $R^p$-module into indecomposables given by
\[ R = \bigoplus_{\substack{0\leq a,b < p \\ a+b \ \text{even}}} x^a y^b R^p  
\ \oplus \ 
\bigoplus_{\substack{0\leq a,b < p \\ a+b \ \text{odd}}} x^a y^b I^p, \]
where $I^p$ is the $R^p$ submodule of $S^p$ consisting of $p$th powers of elements of $I$.

 Write \[F= \bigoplus_{\substack{0\leq a,b < p \\ a+b \ \text{even}}} x^a y^b R^p \ \  \text{and} \ \ A=\bigoplus_{\substack{0\leq a,b < p \\ a+b \ \text{odd}}} x^a y^b I^p.\]
Note that \[ \cJ= \{ \phi \in \End_{R^p}(R) \ | \ \phi(A) \subseteq \m_{{R}^p}R\} \]
is a two-sided ideal: indeed, if $\phi\in 
\cJ$ and $\psi\in \End_{R^p}(R)$, then
\[ \psi \phi (A) \subseteq \psi(\m_{R^p}{R}) \subseteq \m_{R^p}R, \]
and 
\[ \phi \psi(A) \subseteq \phi(\m_{R^p}F \oplus A) \subseteq \m_{R^p}R. \]
Then we have $x^p y\in A$ and \[(xy)^{p-2} x^p y =x^{2p-2} y^{p-1} \notin (x^{2p}, x^p y^p, y^{2p}) = \m_{R^p}R,\]
but since $A\subseteq (x,y)(x^p,y^p)R = (x^{p+1},x^p y,xy^p,y^{p+1})$,
\[ (xy)^{p-1} A \subseteq (x^{2p},x^{2p-1}y^p,x^p y^{2p-1}, y^{2p}) \subseteq \m_{R^p}R.\]
Thus, $(xy)^{p-2}\notin\cJ$, while $(xy)^{p-1}\in \cJ$. Since this is the case for every prime $P$, the hypotheses of Proposition~\ref{prop:charp-to-roots} hold with the polynomial $r(p)=p-2$. We deduce that $-2$ is indeed a root of the sandwich Bernstein-Sato polynomial of $f$.
\end{example}

\begin{example}
Let $R=K[x^2,xy,y^2]$ and $f=x^2$.
For $K$ of characteristic zero, we have $\BS{f}{R}(s) = (s+1)(s+\frac{1}{2})$ following similar lines as the previous example. However, the equality $\SBS{f}{R}(s) = \BS{f}{R}(s)$ holds for this element. It suffices to exhibit a functional equation with this corresponding polynomial $(s+1)(s+\frac{1}{2})$. Indeed, we have
\[  \frac{1}{4} \partial_x^2 f^{s+1} - \frac{1}{2} \partial_x \partial_y f^{s+1} y \partial_x + \frac{1}{2} y\partial_x f^{s+1} \partial_x \partial_y + \frac{1}{4} f^{s+1} \partial_x^2 = (s+1)(s+\frac{1}{2}) f^s. \]
\end{example}

\section{Bernstein's inequality for the Segre product of projective spaces}

We will apply the theory of sandwich Bernstein-Sato polynomials to show that this ring $R$ also satisfies Bernstein's inequality.

\begin{proposition}
For a field $K$ of characteristic zero and
\[ R:=K[x_iy_j\, |\, 1\leq i \leq a, 1\leq j\leq b] \subseteq S:= K[x_1,\ldots, x_a, y_1,\ldots, y_b],\] the element $x_1 y_1$ admits a nonzero sandwich Bernstein-Sato polynomial with roots $-1,-a,-b$.
\end{proposition}
\begin{proof}
Let $K$ be a field of characterstic zero.
First we exhibit a nonzero sandwich Bernstein-Sato equation. We compute for $i,j>1$,
\[ \begin{aligned}
\cA_i &:= \big[ x_i \partial_{x_1} , [ \partial_{x_i} \partial_{y_1}, (x_1y_1)^{s+1}] \big] = (x_1y_1)^s \big( (s+1)^2 x_i \partial_{x_i} - (s+1) x_1 \partial_{x_1}\big),\\
\cB_j &:= \big[ y_j \partial_{y_1} , [ \partial_{y_j} \partial_{x_1}, (x_1y_1)^{s+1}] \big] = (x_1y_1)^s \big( (s+1)^2 y_j \partial_{y_j} - (s+1) y_1 \partial_{y_1}\big).
\end{aligned}
\]
We then have 
\[ \sum_{i=2}^a \cA_i - \sum_{j=2}^b \cB_j = (x_1y_1)^s \big(
(s+1)(s+a) x_1 \partial_{x_1} - (s+1)(s+b) y_1 \partial_{y_1} \big)
 \]
 in $D_R$, using that $\sum_{i=1}^a x_i \partial_{x_i} - \sum_{j=1}^b y_j \partial_{y_j} =0$.

 We compute further
\[\begin{aligned}
\cC := [ \partial_{x_1} \partial_{y_1} , (x_1y_1)^{s+1}] &= (x_1y_1)^s \big((s+1) x_1 \partial_{x_1} + (s+1) y_1 \partial_{y_1} + (s+1)^2\big) \\
\cD:=[ \partial_{x_1} \partial_{y_1} , x_1 \partial_{x_1}(x_1y_1)^{s+1}] &= (x_1y_1)^s \big((s+1)(x_1 \partial_{x_1})^2+(s+2)(x_1 \partial_{x_1})(y_1 \partial_{y_1})\\&\null \qquad\qquad\qquad+2(s+1)^2(x_1 \partial_{x_1}) + (s+1)^2(y_1 \partial_{y_1})+(s+1)^3\big). \\
\end{aligned}\]
Observe that the operators $E:=x_1 \partial_{x_1}$ and $F:=y_1 \partial_{y_1}$ commute in $D_R$. 

One verifies easily that in the polynomial ring $\C[s,A,B,E,F]$, the polynomial\\ ${(s+1)^3(s+A)(s+B)}$ is an element of the ideal generated by 
\[ \begin{aligned} \{(s+1)(s+A) &E - (s+1)(s+B) F, (s+1)E + (s+1) F + (s+1)^2,\\ &(s+1) E^2 + (s+1) EF + 2 (s+1)^2 E + (s+1)^2 F + (s+1)^3\}; \end{aligned}\]
 in particular, the polynomial $(s+1)^3(s+A)(s+B)$ is a right $\C[s,A,B,E,F]$-linear combination of the generators above. Since $x_1 \partial_{x_1}$ and $y_1 \partial_{y_1}$ commute in the ring $D_R$, there is a ring homomorphism from $\C[s,A,B,E,F]$ to $D_R$ mapping $A\mapsto a$, $B\mapsto b$, $E\mapsto x_1 \partial_{x_1}$, and $F\mapsto y_1\partial_{y_1}$. Applying this homomorphism, we obtain that $(x_1 y_1)^s (s+1)^3(s+a)(s+b)$ is a right $D_R$-linear combination of $\cC, \cD, \cA_i\, (2\leq i \leq a),$ and $\cB_j\, (2\leq j \leq b)$. In particular, there is a sandwich Bernstein-Sato functional equation with polynomial $(s+1)^3(s+a)(s+b)$.

Now we claim that $-1,-a,-b$ are all roots of the sandwich Bernstein-Sato polynomial of $x_1y_1$. That $-1$ is a root is clear, e.g., since the Bernstein-Sato polynomial divides the sandwich Bernstein-Sato polynomial. To show that $-a$ and $-b$ are roots, we will apply Proposition~\ref{prop:charp-to-roots}. Henceforth, we let $K$ be a perfect field of characteristic $p>0$.

Set $I_i$ to be the $R$-submodule of $S$ generated by the monomials in the $y$ variables of degree~$i$ and~$J_j$ to be the $R$-submodule of $S$ generated by the monomials in the $x$ variables of degree~$j$. Note that $I_i$ is isomorphic to the ideal $(x_1y_1,\dots,x_1y_b)^i R$, and $J_j$ is isomorphic to the ideal $(x_1y_1,\dots,x_ay_1)^j R$.

We proceed to describe an $R^p$-module decomposition of $R$. We will use tuple notation for monomials: $\alpha$ denotes an $a$-tuple of nonnegative integers, $|\alpha|$ the sum of its entries, $||\alpha||$ its largest entry (and similarly $\beta$ denotes a $b$-tuple, etc.). As $K$-vectorspaces, we have
\[ R = \bigoplus_{|\alpha| = |\beta|} K \cdot x^{\alpha} y^{\beta}, \qquad 
I_i = \bigoplus_{|\alpha| + i = |\beta|} K \cdot x^{\alpha} y^{\beta}, \quad \text{and} \quad J_j = \bigoplus_{|\alpha| = |\beta|+i} K \cdot x^{\alpha} y^{\beta}.\]

We have an $R^p$-module decomposition
\[ R =  F \oplus \bigoplus_{0< i < a} A_i \oplus \bigoplus_{0< j < b} B_j,\]
where
\[ F= \bigoplus_{\substack{|\alpha| = |\beta| \\  ||\alpha||,||\beta||<p}} R^p x^\alpha y^\beta, \qquad
A_i =  \bigoplus_{\substack{|\alpha| = |\beta| + ip \\ ||\alpha||,||\beta||<p}} I_i^p  x^\alpha y^\beta, \quad \text{and} \quad
B_j = \bigoplus_{\substack{ |\alpha| = |\beta| - jp \\ ||\alpha||,||\beta||<p}} J_j^p  x^\alpha y^\beta,\]
and $I_i^p$ and $J_j^p$ denote the $R^p$ submodules of $S^p$ given by the set of $p$th powers of $I_i$ and $J_j$, respectively. 

We claim that $\phi(I_i) \subseteq \m M$ for each $\phi\in \Hom_R(I_i,M)$, where $M= R$, $M=J_j$, or $M=I_k$ with $k<i$. Since each module is fine graded (i.e., graded in the $\N^{a+b}$ grading where each monomial is weighted by its exponent vector), it suffices to show that there is no fine graded map that sends monomial minimal generators to monomial minimal generators. Since any proper quotient of $I_i$ has rank zero and each choice of $M$ is torsionfree, any nonzero map $\phi$ must be injective. If $M=R$ or $M=I_k$ with $k<i$, then $M$ has fewer minimal generators than $I_i$, so no such injective map exists. If $M=J_j$, note the relation $(x_1 y_2) (y_1^i) = (x_1 y_1) (y_1^{i-1}y_2)$  between generators of $I_i$ does not hold for any pair of monomial generators of $J_j$. This establishes the claim.

It follows that \[ \cJ := \{ \phi \in \Hom_{R^p}(R,R) \ | \ \phi(A_{a-1}) \subseteq \m^{[p]}R\}\]  is a two-sided ideal of $\Hom_{R^p}(R,R)$, since for $\phi\in \cJ$ and $\psi\in \Hom_{R^p}(R,R)$ we have
\[ \psi \phi( A_{a-1}) \subseteq \psi (\m^{[p]} R) \subseteq \m^{[p]} R \quad \text{and}  \quad \phi \psi (A_{a-1}) \subseteq \phi(A_{a-1} + \m^{[p]} R ) \subseteq \m^{[p]} R.\]

Assume that $p\geq a$.
Then $(x_1 y_1)^{p-a} \notin \cJ$, since $x_1^{a-1} x_2^{p-1} \cdots x_a^{p-1} y_1^{(a-1)p} \in A_{a-1}$. On the other hand $|\alpha|=|\beta|+(a-1)p$ and $||\alpha||<p$ implies that each coordinate of $\alpha$, in particular the first coordinate, is at least $a-1$. Every monomial of $I_{a-1}^p$ has a $y$-variable with an exponent at least $p$. Consequently, $(x_1 y_1)^{p-a+1}$ times any monomial in $A_{a-1}$ lies in $\m^{[p]}R$, so $(x_1 y_1)^{p-a+1} \in \cJ$.

Thus, by Proposition~\ref{prop:charp-to-roots}, $-a$ is a root of the sandwich Bernstein-Sato polynomial of $x_1 y_1$ over $R$. Similarly, $-b$ is a root as well.
\end{proof}

\begin{remark}
    In the example above, the Bernstein-Sato polynomial of $x_1y_1$ in $R$ equals that in $S$, since $R$ is a differentially extensible direct summand of $S$ in the sense of \cite[Definition~6.1]{BJNB}; see \cite[Lemma~7.3]{BJNB}. Thus, the Bernstein-Sato polynomial is $(s+1)^2$. On the other hand, the sandwich Bernstein-Sato polynomial is either $(s+1)^2(s+1)(s+b)$ or $(s+1)^3(s+1)(s+b)$ so the Bernstein-Sato polynomial and sandwich Bernstein-Sato polynomial differ in this example; we do not know whether the multiplicity of $-1$ as a root here is two or three.
\end{remark}

We conclude by applying the theory developed to establish various finiteness properties of for the ring of differential operators of the Segre product of polynomial rings.

\begin{theorem}\label{thm:Segre}
    Let $a,b>1$ and $R=\C[ x_i y_j \ | \ 1\leq i\leq a, 1\leq j \leq b]$ the the coordinate ring of the Segre embedding of $\P^{a-1}\times \P^{b-1}$. Then

    \begin{enumerate}
        \item The ring $D_{R|\C}$ equipped with any Bernstein filtration $\cBb$ is linearly simple.
        \item If $M$ is a nonzero $D_{R|\C}$-module, and $\cFb$ is a filtration on $M$ compatible with $\cBb$, then $\Dim(M,\cFb)\geq \dim(R)$.
        \item Every nonzero element of $R$ has a nonzero sandwich Bernstein-Sato polynomial.
    \end{enumerate}
\end{theorem}
\begin{proof}
For the first statement, notice first that $R_{x_1y_1}$ is regular. Since $x_1y_1$ has a sandwich Bernstein-Sato polynomial with no nonnegative integer roots, by Theorem~\ref{SBSimpliesLS}, any Bernstein filtration $\cBb$ on $D_{R|K}$ is linearly simple.

Since $R$ is a direct summand of a polynomial ring (e.g., as the zeroth graded piece of $K[x_1,\dots,x_a,y_1,\dots,y_b]$ under the grading that gives each $x$ variable degree one and each $y$ variable degree $-1$), the second statement follows from the first by Corollary~\ref{cor:dir}.

The final statement follows from the first statement by Corollary~\ref{cor:toric}.
\end{proof}

\section*{Acknowledgements}

We thank Luis N\~u\~nez Betancourt and Anurag K.~Singh for comments on a draft of this article.

\newcommand{\etalchar}[1]{$^{#1}$}

\end{document}